\documentclass[a4paper, 9.5pt, final]{amsart}

\usepackage[T1]{fontenc}
\usepackage[english]{babel}
\usepackage{fullpage}
\usepackage{amsmath}
\usepackage{dsfont}\let\mathbb\mathds
\usepackage{setspace}
\usepackage{amsfonts}
\usepackage{amssymb}
\usepackage{graphicx}
\usepackage{amscd}
\usepackage{mathrsfs}
\usepackage{fancybox}

\usepackage{lmodern}
\usepackage{graphicx}
\usepackage{fancyheadings}
\usepackage{chngcntr}
\counterwithout{figure}{section}
\usepackage{tikz}
\usetikzlibrary{calc}
\usepackage{pgfplots}
\usepgfplotslibrary{groupplots}
\usepgfplotslibrary{units}

\usetikzlibrary{arrows}
\usepackage{tkz-berge}

\usetikzlibrary{graphs,shapes.geometric,arrows,fit,matrix,positioning}

\usepackage[extra]{tipa}

\usepackage[OT2,T1]{fontenc}

\usepackage{amsmath, amsthm, amsfonts}
\cornersize{2}
\makeatletter
\newcounter {subsubsubsection}[subsubsection]
\renewcommand\thesubsubsubsection{\thesubsubsection .\@alph\c@subsubsubsection}
\newcommand\subsubsubsection{\@startsection{subsubsubsection}{4}{\z@}%
                                     {-3.25ex\@plus -1ex \@minus -.2ex}%
                                     {1.5ex \@plus .2ex}%
                                     {\normalfont\normalsize\bfseries}}
\newcommand*\l@subsubsubsection{\@dottedtocline{4}{10.0em}{4.1em}}
\newcommand*{\subsubsubsectionmark}[1]{}
\makeatother

\usepackage{graphicx}

\DeclareMathOperator{\coeff}{coeff}

\DeclareMathOperator{\weight}{weight}
\DeclareMathOperator{\id}{id}
\DeclareMathOperator{\RT}{RT}

\DeclareMathOperator{\I}{I}
\DeclareMathOperator{\pr}{pr}
\DeclareMathOperator{\DMR}{DMR}
\DeclareMathOperator{\trf}{trf}

\DeclareMathOperator{\red}{red}
\DeclareMathOperator{\restr}{restr}

\DeclareMathOperator{\colim}{colim}
\DeclareMathOperator{\const}{const}
\DeclareMathOperator{\Ad}{Ad}
\DeclareMathOperator{\NS}{NS}
\DeclareMathOperator{\DM}{DM}
\DeclareMathOperator{\Fil}{Fil}
\DeclareMathOperator{\DMT}{DMT}
\DeclareMathOperator{\Ind}{Ind}
\DeclareMathOperator{\Frac}{Frac}
\DeclareMathOperator{\loc}{loc}

\DeclareMathOperator{\an}{an}

\DeclareMathOperator{\un}{un}
\DeclareMathOperator{\DR}{DR}
\DeclareMathOperator{\Diag}{Diag}
\DeclareMathOperator{\rig}{rig}
\DeclareMathOperator{\Spec}{Spec}
\DeclareMathOperator{\B}{B}

\DeclareMathOperator{\crys}{crys}

\DeclareMathOperator{\Map}{Map}
\DeclareMathOperator{\inv}{inv}

\DeclareMathOperator{\Gal}{Gal}

\DeclareMathOperator{\NX}{NX}
\DeclareMathOperator{\mot}{mot}
\DeclareMathOperator{\per}{per}

\DeclareMathOperator{\Li}{Li}

\DeclareMathOperator{\shft}{shft}

\DeclareMathOperator{\Hom}{Hom}

\DeclareMathOperator{\KZ}{KZ}
\DeclareMathOperator{\MT}{MT}

\DeclareMathOperator{\har}{har}

\theoremstyle{definition}

\newtheorem{Theorem}{Theorem}[section]
\newtheorem{Proposition}[Theorem]{Proposition}

\newtheorem{Lemma}[Theorem]{Lemma}

\newtheorem{Definition}[Theorem]{Definition}

\newtheorem{Proposition-Definition}[Theorem]{Proposition-Definition}

\newtheorem{Fact}[Theorem]{Fact}

\newtheorem{Example}[Theorem]{Example}

\newtheorem{Notation}[Theorem]{Notation}

\newtheorem{Conjecture}[Theorem]{Conjecture}

\newtheorem{Problem}[Theorem]{Problem}
\newtheorem{Corollary}[Theorem]{Corollary}

\newtheorem{Remark}[Theorem]{Remark}
\newtheorem{Nota Bene}[Theorem]{Nota Bene}

\newtheorem{Examples}[Theorem]{Examples}

\numberwithin{equation}{subsection}

\newcommand{\simlra}{\buildrel \sim \over \longrightarrow}

\input cyracc.def
\DeclareFontFamily{U}{russian}{}
\DeclareFontShape{U}{russian}{m}{n}
        { <5><6> wncyr5
        <7><8><9> wncyr7
        <10><10.95><12><14.4><17.28><20.74><24.88> wncyr10 }{}
\DeclareSymbolFont{Russian}{U}{russian}{m}{n}
\DeclareSymbolFontAlphabet{\mathcyr}{Russian}
\makeatletter
\let\@math@cyr\mathcyr
\renewcommand{\mathcyr}[1]{\@math@cyr{\cyracc #1}}
\makeatother
\newcommand{\sh}{\mathcyr{sh}} 

\usepackage[top=4cm, bottom=4cm, left=3.2cm, right=3.2cm]{geometry}

\author{David Jarossay}

\title{}

\address{Institut de Recherche Math\'{e}matique Avanc\'{e}e, Universit\'{e} de Strasbourg, 7 rue Ren\'{e} Descartes, 67000 Strasbourg, France}
\email{jarossay@math.unistra.fr}

\begin{document}

\begin{center}
\begin{Large}\textbf{AN EXPLICIT THEORY OF $\pi_{1}^{\un,\crys}(\mathbb{P}^{1} - \{0,\mu_{N},\infty\})$} \end{Large}
\\ \text{ }
\\ \begin{large} \textbf{II : Algebraic relations and the harmonic Ihara action
\\ \text{ }
\\ II-3 : Sequences of multiple harmonic sums viewed as periods} \end{large}
\end{center}

\maketitle

\begin{abstract}
Let $X=\text{ }\mathbb{P}^{1} - (\{0,\infty\} \cup \mu_{N})\text{ }/\text{ }W(\mathbb{F}_{q})$, with $N \in \mathbb{N}^{\ast}$ and $\mathbb{F}_{q}$ of characteristic $p$ prime to $N$ and containing a primitive $N$-th root of unity. We establish an explicit theory of the crystalline Frobenius of the pro-unipotent fundamental groupoid of $X$.
\newline In part I, we have computed explicitly the Frobenius action. In part II, we use this computation to understand explicitly the algebraic relations of cyclotomic $p$-adic multiple zeta values.
\newline We have used the ideas and the vocabulary of the Galois theory of periods, and in our framework, certain sequences of prime weighted multiple harmonic sums have been dealt with as if they were periods.
\newline In this II-3, we define three notions which essentialize our three types of computations, respectively : a "continuous" groupoid $\pi_{1}^{\un,\DR}(X_{K})^{\widehat{\text{cont}}}$, a "localization" $\pi_{1}^{\un,\DR}(X_{K})^{\loc}$ of $\pi_{1}^{\un,\DR}(X_{K})$, and a "rational counterpart at zero" $\pi_{1}^{\un,\RT,0}(X_{K})$ of
\newline $\pi_{1}^{\un,\DR}(X_{K})$. As an application, and as a conlcusion of this part II, we justify and clarify our Galois-theoretic point of view, in particular, we construct period maps and state period conjectures for sequences of prime weighted multiple harmonic sums.
\end{abstract}

\noindent
\newline
\newline
\newline
\newline
\newline

\tableofcontents

\newpage

\section{Introduction}

\subsection{} Let $p$ be a prime number, $N \in \mathbb{N}^{\ast}$ an integer prime to $p$, and $R = W(\mathbb{F}_{q})$, the ring of Witt vectors of a finite field of characteristic $p$ which contains a primitive $N$-th root of unity. We denote a primitive $N$-th root of unity in $R$ by $\xi$, we denote also by $z_{i} = \xi^{i}$ for $i \in \{1,\ldots,N\}$, and by $K=\Frac(R)$. Let $X$ be the curve $\mathbb{P}^{1} - (\{0,\infty\} \cup \mu_{N})\text{ }/\text{ }R$, and $X_{K}$ be the base change of $X$ to $K$. 
\newline In this work, we construct an explicit theory of the crystalline pro-unipotent fundamental groupoid $\pi_{1}^{\un,\crys}(X)$. In part I, We computed explicitly the Frobenius action. In part II, we use this explicit computations to understand explicitly the algebraic relations of cyclotomic $p$-adic multiple zeta values $\zeta_{p,\alpha}$.
\newline In this II-3, we want to understand intrinsically the nature of the new objects and strategies of proofs that appeared from I-1 to II-2. 

\subsection{} It was granted by general theories of $p$-adic analysis that there was a way to obtain explicit formulas for cyclotomic $p$-adic multiple zeta values, in terms of multiple harmonic sums, where we call weighted multiple harmonic sums the following numbers :
$$ \har_{n} \big(
\begin{array}{c} z_{i_{d+1}},\ldots,z_{i_{1}} \\ s_{d},\ldots,s_{1} \end{array} \big) = n^{s_{d}+\ldots+s_{1}} \sum_{0<n_{1}<\ldots<n_{d}<n}
\frac{\big( \frac{z_{i_{2}}}{z_{i_{1}}} \big)^{n_{1}} \ldots \big(\frac{z_{i_{d+1}}}{z_{i_{d}}}\big)^{n_{d}} 
	\big(\frac{1}{z_{i_{d+1}}}\big)^{n}}{n_{1}^{s_{1}}\ldots n_{d}^{s_{d}}} \in \mathbb{Q}[\xi] $$
\noindent (the (non-weighted) multiple harmonic sums are the same object without the factor $n^{s_{d}+\ldots+s_{1}} = n^{\weight(w)}$, $w=\big(
	\begin{array}{c} z_{i_{d+1}},\ldots,z_{i_{1}} \\ s_{d},\ldots,s_{1} \end{array} \big)$). These numbers are, essentially, the coefficients of the series expansions at $0$ of hyperlogarithms, which are the fundamental solution of $\nabla_{\KZ}$, the canonical connexion on bundle of paths of $\pi_{1}^{\un,\DR}(X_{K})$ starting at a given base-point.
	\newline
\newline However, to our knowledge, nothing was granted regarding the form of the explicit formulas. Moreover, the primary form of such formulas was heavy and unreadable as soon as in depth two, their complexity increase exponentially with respect to the depth, and they did not seem to shed any light on the questions about $\zeta_{p,\alpha}$ that we wanted to solve (listed in \S1 of I-1).
\newline 
\newline It seems to us that the five first papers (I-1 to II-2) of this work provide essentially two messages :
\newline 1) There are actually formulas for $p$-adic multiple zeta values that are readable and exploitable.
\newline 2) The theory can be formulated entirely as analogous to a particular case of a Galois theory of periods, centered around the ideas of "comparing two realizations" of the $\pi_{1}^{\un}$ and exploiting the properties of a "motivic Galois action", except that the words realizations and motivic Galois actions were used in a (semi-)metaphoric sense.

\subsection{} This paper is a formalization of the second message above with, as one of the main technical goals, to define period maps and state period conjectures for prime weighted multiple harmonic sums. Thus, let us recall more precisely the Galois-theoretic formulation of the previous papers.
\newline 
\newline We made a distinction between three types of computations :
\newline - the computations in the "De Rham" (or rigid,  crystalline) setting, using only $p$-adic multiple zeta values.
\newline - the computations in the "De Rham-rational" setting where we viewed multiple harmonic sums as coefficients of the series expansion of hyperlogarithms.
\newline - the computations in the "rational" setting where we viewed multiple harmonic sums only as iterated sums, without any appearent reference to the $\pi_{1}^{\un}$, 
\newline We made as if the computations in the "rational" setting represented a different realization of the $\pi_{1}^{\un}$ and our strategy consisted in relating to each other the computations in these three settings.
\newline 
\newline In I-2 and I-3 we found formulas for expressing cyclotomic $p$-adic multiple zeta values in terms of weighted multiple harmonic sums. They involve a new object called the harmonic Ihara action. We have three  incarnations of the harmonic Ihara action : three group actions $\circ_{\har}^{\DR,\DR}$, $\circ_{\har}^{\DR,\RT}$, $\circ_{\har}^{\RT,\RT}$. They are related to the Ihara action, which is a byproduct of the motivic Galois action on $\pi_{1}^{\un,\DR}(X_{K})$.
We have by I-2, firstly, formulas relating $\zeta_{p,\alpha}$ and $\har$ :
\begin{equation} \label{eq:harmonic torsor 1}\Phi_{p,\alpha} \circ_{\har}^{\DR,\RT} \har_{\mathbb{N}}^{(p^{\alpha})} = \har_{p^{\alpha}\mathbb{N}} 
\end{equation}
\begin{equation} \label{eq:harmonic torsor 2}\har_{p^{\alpha}} \circ_{\har}^{\RT,\RT} \har_{\mathbb{N}}^{(p^{\alpha})} = \har_{p^{\alpha}\mathbb{N}} 
\end{equation}
\noindent and secondly, formulas relating to each other the versions of the harmonic Ihara action, by "comparison maps" ${\Sigma}^{\RT}$ and ${\Sigma}_{\inv}^{\DR}$ satisfying :
\begin{equation} \label{eq:exchange}\Sigma_{\inv}^{\DR} \circ \Sigma^{\RT} = \id \end{equation}
\noindent which imply, by the freeness of the harmonic Ihara action :
\begin{equation} \label{eq:comparison 1}\Phi_{p,\alpha} = {\Sigma}^{\RT} \har_{p^{\alpha}} \end{equation}
\begin{equation} \label{eq:comparison 2}\har_{p^{\alpha}} = {\Sigma}_{\inv}^{\DR} \Phi_{p,\alpha} \end{equation}
\noindent (we will recall in \S6 the definition of each term of these equalities). We have by I-3 other formulas involving $\circ_{\har}^{\DR,\DR}$ obtained by studying the Frobenius as a function of its number of $\alpha$, itself viewed as a $p$-adic integer.
\newline 
\newline In II-1 and II-2 we applied the idea that "the motivic Galois action preserves algebraic relations" to the harmonic Ihara action (and the underlying comparison maps), and by considering the standard families of algebraic relations among multiple zeta values, in particular the double shuffle relations. In II-1 we obtained, in particular, from the three settings variants of the double shuffle relations for sequences of prime multiple harmonic sums, suggesting that they could be viewed as periods, explicit counterparts to $p$-adic multiple zeta values ; in II-2 we obtained two different ways to answer to the question of reading explicitly the quasi-shuffle relation for $p$-adic multiple zeta values.
\newline 
\newline We are going to essentialize independently each of our the three types of computations, $\DR$, $\DR-\RT$ and $\RT$, and at the end of this paper (\S7.4), we will also formalize more globally the analogy with Galois theory of periods. 

\subsection{\label{DR essen}} We will essentialize the computations in the De Rham setting as follows.
\newline Let $\omega_{\DR}$ be the canonical De Rham base point and $\Pi_{\omega_{\DR}} = \pi_{1}^{\un,\DR}(X_{K},\omega_{\DR})$ ;  $\mathcal{O}(\Pi_{\omega_{\DR}})$ is a graded Hopf algebra over $\mathbb{Q}$, where the grading is the weight. Let us choose a convenient increasing ring filtration $\Fil$ over $\mathbb{Q}$, representing the authorized rational coefficients at each weight : we obtain  $\mathcal{O}(\Pi_{\omega_{\DR}})^{\Fil}$, a filtered module over the filtered ring $(\mathbb{Q},\Fil)$ ; then, it makes sense to consider the continuous points of it (they are characterized by certain bounds on the valuation of the coefficients) which defines a topological group $\Pi_{\omega_{\DR}}^{\text{cont}}$, from which we can take the completion $\Pi_{\omega_{\DR}}^{\widehat{cont}}$.
\newline This construction can be transported at any couple of base-points $(x,y)$ by the canonical isomorphisms of pro-affine schemes  $\pi_{1}^{\un,\DR}(X_{K},\omega_{\DR}) \simeq \pi_{1}^{\un,\DR}(X_{K},\omega_{\DR})$.
\newline One can also give a variant of this construction which takes into account the depth filtration, which is more adapted to our purposes for technical reasons.
\newline 
\newline The variant $\pi_{1}^{\un,\DR}(X_{K})^{\widehat{cont}}$ of the groupoid $\pi_{1}^{\un,\DR}(X_{K})$ that we obtain is the natural receptacle of infinite sums of $p$-adic multiple zeta values, and thus, in particular, by (\ref{eq:comparison 1}), (\ref{eq:comparison 2}) of prime weighted multiple harmonic sums and the absolutely convergent infinite sums of them.
\newline What is interesting is that the Ihara action induces an Ihara action on $\pi_{1}^{\un,\DR}(X_{K})^{\widehat{cont}}$ ; and $\pi_{1}^{\un,\DR}(X_{K})^{\widehat{cont}}$ is also a natural receptacle for the harmonic Ihara actions which were defined through infinite summations and limits.
\newline As the main application, we are going to see that $\pi_{1}^{\un,\DR}(X_{K})^{\widehat{cont}}$ is also a framework for defining motivic prime weighted multiple harmonic sums $\har_{\mathcal{P}^{\mathbb{N}}}^{\mot}$, construct period maps and state period conjectures, and gives rise to a notion of "completed periods".

\subsection{\label{DR-RT essen}} We will essentialize the computations in the "De Rham-rational setting" by defining an object called a localization the De Rham fundamental groupoid. The term localization refers to the following thing. The Hopf algebra $\mathcal{O}(\Pi_{\omega_{\DR}})$ is the shuffle Hopf algebra $\mathcal{O}^{\sh,e_{Z}}$ over the alphabet $e_{Z} = \{e_{0},e_{z_{1}},\ldots,e_{z_{N}}\}$. Its product is the shuffle product $\sh$, but it is usual to embed it into the universal enveloping algebra of its Lie algebra and to consider the product defined by the concatenation of words. Then, by the usual comparison isomorphisms, each $e_{z_{i}}$ can be viewed as an integration operator $ e_{z_{i}} \leftrightarrow \big( f \mapsto \int \frac{dz}{z-z_{i}}f \big) $, over any appropriate space of functions, and a word over $e_{Z}$ can be viewed as the iterated integration operator obtained by composing these operators, by reading the letters of the word from the right to the left. The localization will include the possibility to add formal inverses of the letters of $e_{Z}$ for the concatenation product : $e^{-1}_{z_{i}} \leftrightarrow \big( f \mapsto (z-z_{i})\frac{df}{dz} \big)$.
\newline This will provide to different applications, for indexing multiple harmonic sums.
\newline Independently from this construction we will write a conjecture of periods for sequences of $\har_{p^{\alpha}}$, essentializing our computations in the De Rham-rational setting.

\subsection{\label{RT essen}} What we deal with in this theory is only the rational setting "at $0$", i.e. the coefficients of the series expansions at $0$. The series expansions of iterated integrals at all points, and for more general varieties, will be studied in another paper.
\newline The "rational setting at $0$" can be expressed in terms of iterated sums, as the one representing multiple harmonic sums. They satisfy a difference equation which reflects the $\KZ$ equation. What we will define is an analogue for iterated sums of the usual setting for dealing with iterated integrals.
\newline We will also formulate a conjecture of periods which imitates the conjecture of periods of Kontsevich-Zagier, but with standard operations on multiple harmonic sums, and which essentializes our computations in the rational setting.

\subsection*{Outline}

We review the motivic framework of cyclotomic multiple zeta values in \S2. In \S3, \S4 and \S5 we define three notions which essentialize respectively the "De Rham", "De Rham-rational" and "rational" computations, i.e. the continuous completion of $\pi_{1}^{\un,\DR}(X_{K})$ (\S3) the localization of $\pi_{1}^{\un,\DR}(X_{K})$ (\S4), and the rational counterpart at $0$ of $\pi_{1}^{\un,\DR}(X_{K})$ (\S5). In \S6 we reevaluate the notion of harmonic Ihara action, partly using these notions. The conclusion is \S7 where we construct period maps and state some period conjectures for sequences of prime weighted multiple harmonic sums ; we also define other notions making more concrete the analogy with Galois theory of periods.
\newline The first appendix is a comment on the methods to obtain variants of the standard algebraic relations of multiple zeta values. The second appendix is about the question of the transcendence of sequences of multiple harmonic sums.
\newline 
\newline Unlike for the five previous papers (I-1 - II-2), the main results of this paper are not computations but definitions and conjectures (although the proofs that certain objects are well-defined sometimes rely on short computations). In other papers, we will define the three notions above from \S3, \S4, \S5 for general varieties.

\subsection*{Acknowledgments}

This work has been achieved at Institut de Recherche Math\'{e}matique Avanc\'{e}e, Strasbourg, supported by Labex IRMIA. I thank Benjamin Enriquez and Pierre Cartier for their support.

\section{Review of motivic cyclotomic multiple zeta values}

\subsection{The motivic $\pi_{1}^{\un}$ of $\mathbb{P}^{1} - \{z_{0},\ldots,z_{r}\}$ over a number field}

The goal of this paragraph is to review the definition of the motivic pro-unipotent fundamental groupoid. A first notion of $\pi_{1}^{\un}$ is defined in \cite{Deligne},\S13 ; a second one in \cite{Goncharov},\S4 ; and finally a third one in \cite{Deligne Goncharov}. We try to summarize the construction in \cite{Deligne Goncharov}. It requires to recall first some facts on Tannakian categories and on categories of motives.

\subsubsection{Affine schemes in a Tannakian category}

Let $\mathcal{T}$ a Tannakian category ; let $\Ind \mathcal{T}$ be the category of Ind-objects of $\mathcal{T}$. A unitary commutative algebra in $\Ind \mathcal{T}$ is an object $A$ with a product $A \otimes A \rightarrow A$ and a unity $1 \rightarrow A$. The category of $T$-affine schemes is the dual of the category of unitary commutative algebras in $T$ ; $\Spec(A)$ is the $T$-affine scheme associated with the algebra $A$. (More details in \cite{Deligne Goncharov}, \S2.6).

\subsubsection{Tannakian categories of mixed Tate motives\label{U omega}}

Let us recall some notations and facts on categories of motives (more details and references are given in \cite{Deligne Goncharov}, \S1,\S2) : for $k$ a field of characteristic zero, $\DM(k)$ is the triangulated category of motives of Levine and Voevodsky ; $\DMT(k)$ is its triangulated subcategory generated by Tate objects ; $\text{SmCor}(k)$ is the category of smooth correspondences over $k$ (\cite{Deligne Goncharov}, \S1.5) ; $\MT(k)$ is the Tannakian category of mixed Tate motives over $k$ and $\omega$ is its canonical fiber functor ; if $S$ is a finite set of finite places of $k$, and $\mathcal{O}_{S}$ is the set of $S$-integers of $k$, $\MT(\mathcal{O}_{S})$ is the Tannakian category of mixed Tate motives over $\mathcal{O}_{S}$ and $\omega$ is its canonical fiber functor.
\newline The scalar extension of $\omega$ from $\mathbb{Q}$ to $k$ is canonically isomorphic to the De Rham realization functor of $\MT(k)$ (\cite{Deligne Goncharov}, Proposition 2.10). The Galois group $G_{\omega}$ associated with $\omega$ on $\MT(\mathcal{O}_{S})$ admits a semi-direct product decomposition 
$$ G_{\omega} = \mathbb{G}_{m} \ltimes U_{\omega} $$
\noindent where $U_{\omega}$ is pro-unipotent and can be described explicitly (\cite{Deligne Goncharov}, Proposition 2.2, Proposition 2.3). The motivic Galois action of $\mathbb{G}_{m}$ reflects the weight grading ($\lambda \mapsto $multiplication by $\lambda^{\weight}$).

\subsubsection{The motivic $\pi_{1}^{\un}$ and the mixed Tate case}

We follow \cite{Deligne Goncharov},\S3. Let $k$ a field of characteristic zero.
\newline
\newline a) The first ingredient of the construction is the relation between the Betti pro-unipotent fundamental groupoid of a topological space $X$ and the cohomology of $X^{n}$ relatively to certain divisors for all $n \in \mathbb{N}^{\ast}$.
\newline Let $X$ be a connected and locally simply connected topological space. Let $a \in X$, and $\Gamma$ be the group algebra $k[\pi_{1}^{\text{top}}(X,a)]$ over $k$. The structure of $\Gamma$-left-module of $\Gamma/I^{n}$ defines a local system $(A/I^{n}){\tilde{}}$. The unit map of $\Gamma$ induces the unit map $k \rightarrow (\Delta/I^{n+1}){\tilde{}}_{a}$, $\gamma \mapsto \gamma.1$, where $(\gamma/I^{n}){\tilde{}}_{a}=\gamma/I^{n+1}$ is the fiber of $(\Lambda/I^{n+1})\tilde{}$ at $a$. For each $n \in \mathbb{N}$, let $Y_{0},\ldots,Y_{n}$ be the sequence of subsets of $X^{n}$ made of the elements $(t_{1},\ldots,t_{n})$ such that we have, respectively, $b=t_{1}$, $t_{1} =t_{2}$,$\ldots$,$t_{n-1}=t_{n}$, $t_{n}=a$. One has a complex of sheaves over $X^{n}$ :
$$ \underline{k} \rightarrow  \oplus_{|I|=1} k_{I} \rightarrow \ldots \rightarrow \oplus_{|I|=u} k_{I} \ldots \rightarrow \oplus_{|I|=n} k_{I} $$
\noindent where $\underline{k}$ is the constant sheaf $k$, $k_{I}$ is the constant sheaf on $Y_{I}=\cap_{i\in I} Y_{i}$ extended by $0$ on $X^{n}$, and where the differentials are the sums $\sum_{I = \{i_{1}<\ldots<i_{k}\} \subset \{1,\ldots,n\}} \sum_{m=1}^{k} (-1)^{p} \times$ (morphisms induced by the inclusion $k_{J} \subset k_{I}$ for $J \subset I$). Let ${}_b \mathcal{K}_{a}$ be this complex with the last term omitted. When $a\not=b$ we have $\mathbb{H}^{\ast}(X^{n},{}_b \mathcal{K}_{a}) = H^{\ast}(X^{n} ,\cup Y_{i},k)$. When $a=b$, the last differential induces a morphism of complexes ${}_a \mathcal{K}_{a} \rightarrow k_{Y_{\{1,\ldots,n\}}}[-n]$ that induces a map $\mathcal{H}^{\ast}(X^{n},{}_b \mathcal{K}_{a}) \rightarrow k$

\begin{Theorem} (Beilinson, unpublished : \cite{Deligne Goncharov}, Proposition 3.4)
	\newline i) For $i<n$, $\mathbb{H}^{i}(X^{n},{}_b \mathcal{K}_{a})=0$.
	\newline ii) The local system $\mathbb{H}^{\ast}(X^{n},{}_b \mathcal{K}_{a})$ on $b$ on $X$ equipped with its unit is the dual of the local system $(\Lambda/I^{n+1})\tilde{}$ equipped with its unit.
\end{Theorem}

\noindent b) One needs a complex that computes the hypercohomology of ${}_b \mathcal{K}_{a}$, compatibly with the projection maps $\Lambda/I^{n+1} \rightarrow \Lambda/I^{n}$. 
\newline A simplicial object in an additive category is a contravariant functor from the category whose objects are the ordered sets $\Delta_{n} = \{0,1,\ldots,n\}$ to this category. If $X$ is a topological space, for each $a,b \in X$, there is a natural way of associate a cosimplicial space $\Delta_{n} \mapsto X^{n}$ (\cite{Deligne Goncharov}, \S3.6). Then, let $(S_{n})$ be a complex that computes the cohomology of $X^{n}$ and that is functorial for the cosimplicial morphisms between the $X^{n}$'s. They provide a resolution of ${}_b \mathcal{K}_{a} \langle n \rangle$ over $X^{n}$. The complex $\Gamma(X^{n},\text{resolution of } {}_b \mathcal{K}_{a}\text{ defined by }S_{\bullet})$ computes the hypercohomology of ${}_b \mathcal{K}_{a}$.
\newline 
\newline c) One wants to show that the complex of b) arises from an object in a category of motives. Before that, we must rewrite it differently. A system of coefficients on $\Delta_{n}$ with values in an additive category is a contravariant functor which maps each face $F \subset \Delta_{n}$ to an object $c(F)$. A simplicial object $s_{\bullet}$ defines the system of coefficients $c$ defined by, for all $F$, $c(F)=s_{F}$. A system of coefficients $c$ defines in turn a chain complex $C_{\ast}(\Delta_{n},c)$ by $C_{p}(\Delta_{n},c) = \oplus_{|F|=p+1} c(F)$.
\newline The simple complex associated with the double complex $C_{\ast}(\Delta_{n},S_{\ast})[-n]$ can be identified to the complex of b), and the Theorem in a) can be rewritten as 
$$  k[\pi_{1}^{\text{top}}(X,b,a)] \otimes_{\Gamma} \Gamma/I^{n+1} = H^{0} C_{\ast}(\Delta_{n},S_{\bullet}^{\ast}) $$
\noindent As explained in \cite{Deligne Goncharov}, \S3.11, this complex is a variant of Chen's bar complex of \cite{Chen}.
\newline For certain purposes, it is convenient to replace the complex $ C_{\ast}(\Delta_{n},S_{\bullet}^{\ast})$ by the functorially homotopic (\cite{Deligne Goncharov}, Proposition 3.10) complex $\sigma_{\geq -n} \NS_{\bullet}$, where $N$ refers to the normalized complex and $\sigma_{\geq -n}$ is the truncation to degrees $\geq -n$.
\newline 
\newline d) (\cite{Deligne Goncharov}, \S3.12) Now, let us take $X$, not a suitable topological space, but a connected algebraic variety over a field $K$, and $(a,b) \in X(K)$, we can attach to them a cosimplicial scheme $\Delta_{n} \rightarrow X^{n}$. One can then view it as a simplicial object of the additive category of smooth correspondences $\text{SmCor}(K)$, and, for any $n$, consider the complex ${}_b \Omega_{a}^{[n]}(X) = C_{\ast}(\Delta_{n},X^{\ast})[n]$ ; in the Karoubian envelope $\text{SmCor}_{\text{Kar}}(K)$ of $\text{SmCor}(K)$, the complex $\sigma_{\geq -n} \NX^{\ast}$ is again functorially homotopic to $C_{\ast}(\Delta_{n},X^{\ast})[n]$ by  \cite{Deligne Goncharov}, Proposition 3.10.
Let ${}_b \Omega_{a}^{[n]}(X)$ be the object of $\DM(K)$ arising from one of these complexes (they are naturally isomorphic).
\newline The Betti realization is compatible with the previous construction in the topological case.
\newline 
\newline e) (\cite{Deligne Goncharov}, \S3.12) From now on, $X$ is assumed to be mixed Tate, i.e. $X$ viewed as an object of $\DM(K)$, is an object of $\DMT(K)$. Then the ${}_b \Omega_{a}^{[n]}(X)$ are in $\DMT(K)$. When $K$ is a number field, by the Beilinson-Soul\'{e} conjectures one can obtain an object of $\MT(K)$ as $H^{0} ({}_b \Omega_{a}^{[n]}(X) \otimes \mathbb{Q})^{\vee}$.
\newline One defines the algebra of functions of the fiber at $(b,a)$ of the motivic $\pi_{1}^{\un}(X)$ as $${}_b A_{a} = \colim H^{0} ({}_b \Omega_{a}^{[n]}(X) \otimes \mathbb{Q})^{\vee}$$ 
\noindent which has the structure of a commutative algebra in the category $\text{Ind}\MT(K)$ ; we also have morphisms ${}_c A_{a} \rightarrow  {}_c A_{b} \otimes {}_b A_{a}$.

\subsection{Application to $\mathbb{P}^{1} - \{z_{0},\ldots,z_{r}\}$ over a number field}

We consider $Y = \mathbb{P}^{1} - \{z_{0},\ldots,z_{r}\}$ over a number field.

\subsubsection{Review of $\pi_{1}^{\un,\DR}(Y)$}

The pro-unipotent fundamental groupoid $\pi_{1}^{\un,\DR}(Y)$ has as base-points the rational points of $Y$, and also the non-zero rational tangent vectors at $0,z_{1},\ldots,z_{r},\infty$, and the canonical base-point $\omega_{\DR}$. For all couples of base-points $(x,y)$, we have canonical isomorphisms $\pi_{1}^{\un,\DR}(Y,\omega_{\DR}) \simeq  \pi_{1}^{\un,\DR}(Y,x,y)$ that are compatible with the groupoid structure.
\newline 
\newline The pro-unipotent algebraic group 
$\pi_{1}^{\un,\DR}(\mathbb{P}^{1} - \{0,\mu_{N},\infty\},\omega_{\DR})$ is $\Spec(\mathcal{O}^{\sh,e_{Z}})$, where $\mathcal{O}^{\sh,e_{Z}}$ is the shuffle Hopf algebra over the alphabet $e_{Z} = \{e_{0},e_{z_{1}},\ldots,e_{z_{N}} \}$ ; the product of $\mathcal{O}^{\sh,e_{Z}}$ is the shuffle product $\sh$. We have a (functorial) embedding, for all $\mathbb{Q}$-algebra $A$, $\pi_{1}^{un,\DR}(A) \subset A\langle \langle e_{Z} \rangle\rangle$, where $A\langle \langle e_{Z} \rangle\rangle$ is the non-commutative algebra of power series over the letters of $e_{Z}$ ; $\pi_{1}^{un,\DR}(A)$ is the group of series that are grouplike for the shuffle coproduct, dual to the shuffle product.

\begin{Notation}
	i) $\Pi_{\omega_{\DR}} = \pi_{1}^{\un,\DR}(Y,\omega_{\DR})$
	\newline ii) $\Pi_{z',z} = \pi_{1}^{\un,\DR}(Y,\vec{1}_{z'},\vec{1}_{z})$ for $z,z' \in \{0,z_{1},\ldots,z_{r},\infty\}$.
\end{Notation}

\begin{Definition} For $z \in  \{0,z_{1},\ldots,z_{r},\infty\}$, let $\tilde{\Pi}_{z,0}$ the sub-group scheme whose points are the grouplike series satisfying $f[e_{0}] = f[e_{z}]=0$.
\end{Definition}

\subsubsection{The groupoid $\pi_{1}^{\un,\mot}(Y)$}

The previous theorems concerning the motivic fundamental groupoid in the mixed Tate situation apply when the variety $Y$ is of the form $\mathbb{P}^{1} - Z$ over a number field $k$ ($Z$ a finite number of points. Thus we have a groupoid $\pi_{1}^{\un,\mot}(Y)$ and motivic Galois groups associated with the Betti and De Rham fiber functors of $\MT(k)$ act on $\pi_{1}^{\un,\B}(Y)$ and $\pi_{1}^{\un,\DR}(Y)$ (\cite{Deligne Goncharov}, \S3, \S4.) We recall that the motivic Galois group associated with the De Rham fiber functor has a semi-direct product decomposition : $G_{\omega} = \mathbb{G}_{m} \ltimes U_{\omega}$.

\subsubsection{The motivic Galois action of $\mathbb{G}_{m}$ on $\pi_{1}^{\un,\DR}(Y)$}

The action of $\mathbb{G}_{m}$ on $\pi_{1}^{\un,\DR}(Y)$ reflects the weight grading and is the following : for all $\mathbb{Q}$-algebra $A$ :
 $$ \tau : (\lambda,f) \in \mathbb{G}_{m}(A) \times A \langle\langle e_{Z} \rangle\rangle \mapsto 
 \sum_{w\in\mathcal{W}(e_{Z})} \lambda^{\weight(w)} f[w]w $$

\subsubsection{The action of $U^{\omega}$ on $\pi_{1}^{\un,\DR}(Y)$ \label{Goncharov}}
 
Since the Hodge realization functor of $\MT(k)$ is fully faithful (\cite{Deligne Goncharov}, Proposition 2.14), the action on $U^{\omega}$ can be computed in the Hodge realization. In this framework Goncharov has defined motivic hyperlogarithms (\cite{Goncharov}, Definition 5.5) and computed of the action of $U^{\omega}$ on them (using the version of $\pi_{1}^{\un,\mot}$ constructed in \cite{Goncharov} \S4.3). In this setting, the motivic lift of $\zeta(2)$ is zero. The formula for this action contains in particular the formula for the action of $U^{\omega}$ on $\pi_{1}^{\un,\DR}(X)$. For $\mathbb{P}^{1} - \{0,1,\infty\}$, using the previous formula, one can give another definition of $\zeta^{\mot}$, for which $\zeta^{\text{mot}}(2)\not= 0$ \cite{Brown mixed Tate}. One retrieves Goncharov's definition by moding out by the non-zero $(\zeta^{\text{mot}}(2))$. This can be adapted easily to any $\mathbb{P}^{1} - \{0,\mu_{N},\infty\}$. 
\newline 
\newline The formula for the action of the pro-unipotent part $U_{\omega}$ is the following. According to \cite{Goncharov}, let $I(a_{n+1},a_{n},\ldots,a_{1};a_{0})$ be an iterated integral $\int_{a_{0}}^{a_{n+1}} \omega_{a_{n}} \ldots \omega_{a_{1}}$ where $\omega_{a_{i}}(z) = \frac{dz}{z - a_{i}}$ (see \cite{Goncharov} for more details), and let $\tilde{I}(a_{n+1},a_{n},\ldots,a_{1};a_{0})$ be the Hodge-Tate structure lifting it (\cite{Goncharov}, definition 5.5) resp. its lift in the sense of \cite{Brown mixed Tate}.

\begin{Theorem} (Goncharov, \cite{Goncharov}, theorem 6.4, \cite{Brown mixed Tate}) The Tannakian coproduct on $\tilde{I}$ is given by
	\begin{multline} \Delta \tilde{I} (a_{n+1};a_{n},\ldots,a_{1};a_{0})
	\\ = \sum_{0=i_{0}<i_{1} < \ldots < i_{k} < i_{k+1} = n} 
	\tilde{I}(a_{n+1};a_{i_{k}},\ldots,a_{i_{1}};a_{0}) 
	\otimes 
	\prod_{l=0}^{k} \tilde{I}(a_{i_{p+1}};a_{i_{p+1}-1}\ldots a_{i_{p}+1};a_{i_{p}})
	\end{multline}
\end{Theorem}
\noindent In this paragraph let us denote by $\Pi_{1,0} = \pi_{1}^{\un,\DR}(Y,-\vec{1}_{1},\vec{1}_{0})$. In the case of $\mathbb{P}^{1} - \{0,\mu_{N},\infty\}$, in particular, the coproduct above can be thought of as a map
$$ \mathcal{O}(\Pi_{1,0})
\rightarrow 
\mathcal{O}(\Pi_{1,0}) 
\otimes
\mathcal{O}(\Pi_{1,0}) $$ 
\noindent and composing it with, on the right hand side, the map $\id \otimes \big( \mathcal{O}(\Pi_{1,0}) \rightarrow \mathcal{O}(U_{\omega})\big)$ given by the action of $\mathcal{O}(U_{\omega})$ on the canonical path $1$, we obtain the coaction :
$$ \mathcal{O}(\Pi_{1,0}) 
\rightarrow 
\mathcal{O}(\Pi_{1,0}) 
\otimes 
\mathcal{O}(U_{\omega}) $$
\noindent In \cite{Deligne Goncharov}, \S5 it is explained how to retrieve the Ihara action as a byproduct of the action of $U_{\omega}$. The formulas for the Ihara action and the Goncharov coaction are essentially dual to each other.

\subsubsection{Period maps\label{Yamashita}}

Let us recall that, in general, the algebraic relations between hyperlogarithms involve several curves $\mathbb{P}^{1} - Z$ at the same time ; this phenomenon does not happen only for the curves $\mathbb{P}^{1} - \{0,\mu_{N},\infty\}$. Let us thus consider $\mathbb{P}^{1} - \{0,\mu_{N},\infty\}$.
\newline 
\newline Let $\mathcal{Z}$, resp. $\mathcal{Z}^{\text{mot}}$ be the algebra of cyclotomic multiple zeta values, resp. their motivic lifts. We have a morphism $\mathcal{Z}^{\mot} \rightarrow \mathcal{Z}$ that sends  $\zeta^{\mot}(w) \mapsto \zeta(w)$ (\cite{Brown mixed Tate} for $N=1$). The conjecture of periods for motivic cyclotomic multiple zeta values amounts to say that it is an isomorphism.
\newline
\newline Let $\zeta_{p}^{\KZ}(w)$ be Furusho's $p$-adic multiple zeta values \cite{Furusho 1}, \cite{Furusho 2}, and their generalization from $N=1$ to any $N$ by Yamashita \cite{Yamashita}. Let $\mathcal{Z}_{p}^{\KZ}$ be the $\mathbb{Q}$-algebra that they generate, and let $\mathcal{Z}_{p,\alpha}$ be the $\mathbb{Q}$-algebra generated by the $p$-adic multiple zeta values $\zeta_{p,\alpha}(w)$.
\newline Yamashita has constructed (unpublished) a morphism $\mathcal{Z}^{\mot} \rightarrow \mathcal{Z}_{p}^{\KZ}$, $\zeta^{\mot}(w) \mapsto \zeta_{p}^{\KZ}(w)$ ; one can deduce from it, for any $\alpha \in \mathbb{N}^{\ast}$, a morphism $\mathcal{Z}^{\mot} \rightarrow \mathcal{Z}_{p,\alpha}$ sending $\zeta^{\mot} \mapsto \zeta_{p,\alpha}$.

\section{The (completed) continuous groupoid $\pi_{1}^{\un,\DR}(X_{K})^{\text{cont}}$}

The following concept essentializes our computations in the De Rham setting, and will provide a way to state a period conjecture for sequences of prime weighted multiple harmonic sums in \S\ref{periods}.1. In this section, $X$ is a variety $\mathbb{P}^{1} - \{z_{0},z_{1},\ldots,z_{r}\}$ over a number field.

\subsection{Definition}

The groupoid $\pi_{1}^{\un,\DR}(X_{K})$ is a groupoid of pro-affines schemes over $\mathbb{Z}$. We are going to consider its points with values in filtered rings. 

\subsubsection{Ring filtrations on $\mathbb{Q}$}

\begin{Definition}
	We will call, for simplicity "filtration on $\mathbb{Q}$" an increasing ring filtration $\Fil = (\Fil_{s})_{s \in \mathbb{N}}$ ($\Fil_{s}\Fil_{t}=\Fil_{s+t}$, $\Fil_{s} \subset \Fil_{s+1}$) by sub-$\mathbb{Z}$-modules such that
	$\Fil_{0}=\mathbb{Z}$.
\end{Definition}

\begin{Examples} The natural example, which will reappear in part III, is $\Fil=(\frac{1}{s!}\mathbb{Z})_{s \in \mathbb{N}}$.
\end{Examples}

\noindent For the rest of this \S3, we fix any filtration $\Fil$ as above, and all the next statements will be relative to it.

\subsubsection{The continuous $\pi_{1}^{\un,\DR}$ at the canonical base-point}

Usually we view $\mathcal{O}(\Pi_{\omega_{\DR}})$ as a $\mathbb{Q}$-algebra, graded by the weight. The weight defines in particular a filtration by $W_{s}$

\begin{Definition} i) For any filtration $\Fil$ as above, let $\mathcal{O}(\Pi_{\omega_{\DR}})^{\Fil}$ be $\mathcal{O}(\Pi_{\omega_{\DR}})$ viewed as a filtered module over the filtered ring $(\mathbb{Q},\Fil)$.
	\newline ii) For any $d \in \mathbb{N}^{\ast}$, we  $\mathcal{O}(\Pi_{\omega_{\DR}})^{\Fil}_{\leq d}$ the subring of $\mathcal{O}(\Pi_{\omega_{\DR}})^{\Fil}$ generated by words of depth at most $d$.
	\end{Definition}
	
	\begin{Remark} The ii) in the definition above is motivated by technical reasons from part I. However, the notion of depth is not only a technical notion which works for $\mathbb{P}^{1} - \{z_{0},z_{1},\ldots,z_{r}\}$ : since it is defined by counting, in a sequence of differential forms $\omega_{i_{n}}\ldots \omega_{i_{1}}$, the number of those who have no singularity at the chosen base-point $0$, it is a particular case of a more general notion which applies to the $\pi_{1}^{\un}$ a general variety. Thus, the next notions can be easily adapted for general varieties (including when the canonical base-point $\omega_{\DR}$ does not exist).
	\end{Remark}
	
	\noindent From now on we consider the points of $\Pi_{\omega_{\DR}}$ with rings $A$ equipped with a decreasing separated ring filtration $(A_{s})_{s \geq 0}$ ($A_{s+1} \subset A_{s}$, $\cap_{s} A_{s} = 0$, $A_{s}A_{t} \subset A_{s+t}$), and the separated topology associated with this filtration, we will say "filtered rings" for simplicity.
		
	\begin{Definition} Let $A$ be a filtered ring. We denote by $\Pi_{\omega_{\DR}}(A)^{\Fil} \subset \Pi_{\omega_{\DR}}(A)$ the subgroup of morphisms of filtered rings
		$\mathcal{O}(\Pi_{\omega_{\DR}})^{\Fil} \rightarrow A$.
		\end{Definition}

\begin{Definition} Let $A$ be a filtered ring.
\newline i) Let $\Pi_{\omega_{\DR}}(A)^{\text{cont}} \subset \Pi_{\omega_{\DR}}(A)$ be the subgroup of filtered morphisms $\mathcal{O}(\Pi_{\omega_{\DR}})^{\Fil} \rightarrow A$ that are uniformly continuous.
\newline ii) Let $\Pi_{\omega_{\DR}}(A)^{\text{d-cont}} \subset \Pi_{\omega_{\DR}}(A)$ be the subgroup of maps $\mathcal{O}(\Pi_{\omega_{\DR}})^{\Fil} \rightarrow A$ whose restrictions to $\mathcal{O}(\Pi_{\omega_{\DR}})^{\Fil}_{\leq d}$ are uniformly continuous for all $d \in \mathbb{N}^{\ast}$.
\end{Definition}

\noindent We recall that, for $E,F$ two metric spaces, the space of uniformly continuous functions $E \rightarrow F$ endowed with the uniform distance is complete and isometric to the one of uniformly continuous functions 
$\hat{E} \rightarrow \hat{F}$ where $\hat{E}$ and $\hat{F}$ are the completions of $E$ and $F$. Thus, we have 
$\Pi_{\omega_{\DR}}(A)^{\text{cont}} = \Pi_{\omega_{\DR}}(A)^{\widehat{\text{cont}}}$
and $\Pi_{\omega_{\DR}}(A)^{\text{d-cont}} = \Pi_{\omega_{\DR}}(A)^{\widehat{\text{d-cont}}}$ with the following definition :

\begin{Definition} Let $A$ be a filtered ring and $\hat{A}$ its completion.
	\newline i) Let $\Pi_{\omega_{\DR}}(A)^{\widehat{\text{cont}}} \subset \Pi_{\omega_{\DR}}(A)$ be the subgroup of maps $\widehat{\mathcal{O}(\Pi_{\omega_{\DR}})^{\Fil}} \rightarrow \widehat{A}$ that are uniformly continuous.
	\newline ii) Let $\Pi_{\omega_{\DR}}(A)^{\widehat{\text{d-cont}}} \subset \Pi_{\omega_{\DR}}(A)$ be the subgroup of maps  $\widehat{\mathcal{O}(\Pi_{\omega_{\DR}})^{\Fil}} \rightarrow \widehat{A}$ whose restrictions to $\widehat{\mathcal{O}(\Pi_{\omega_{\DR}})^{\Fil}_{\leq d}}$ are uniformly continuous for all $d \in \mathbb{N}^{\ast}$.
\end{Definition}

\noindent Since our explicit $p$-adic theory of $\pi_{1}^{\un}(\mathbb{P}^{1}- \{0,\mu_{N},\infty\})$ is centered around infinite sums of $\zeta_{p,\alpha}$ (such as $\har_{p^{\alpha}}$, by equation (\ref{eq:comparison 2}), and infinite sums of $\har_{p^{\alpha}}$), we will use the notations $\Pi_{\omega_{\DR}}(A)^{\widehat{\text{cont}}}$ and $\Pi_{\omega_{\DR}}(A)^{\widehat{\text{d-cont}}}$. The situation will be similar later for other applications of this notion.

	\begin{Lemma} For each $A$,  $\Pi_{\omega_{\DR}}(A)^{\widehat{\text{cont}}}$ and $\Pi_{\omega_{\DR}}(A)^{\widehat{\text{d-cont}}}$ are a subgroup for the De Rham product $\times^{\DR}$, and with the uniform topology, they are complete topological groups.
	\end{Lemma}
	
\begin{proof}
	Same with I-2.
	\end{proof}

\subsubsection{The continuous $\pi_{1}^{\un,\DR}$ at all base-points}

\begin{Definition} For each couple of base-points, we define $\pi_{1}^{\un,\DR}(X_{K},y,x)(A)^{\widehat{\text{cont}}}$, $\pi_{1}^{\un,\DR}(X_{K},y,x)(A)^{\widehat{\text{d-cont}}}$, via the previous definitions and the canonical isomorphism 
	$\pi_{1}^{\un,\DR}(X_{K},y,x) \simeq \Pi_{\omega_{\DR}}$.
\end{Definition}

\noindent Since the canonical isomorphisms 	$\pi_{1}^{\un,\DR}(X_{K},y,x) \simeq \Pi_{\omega_{\DR}}$ are compatible with the groupoid structure, we obtain in particular maps 
$$  \pi_{1}^{\un,\DR}(X_{K},z,y)(A)^{\widehat{\text{cont}}}\times  \pi_{1}^{\un,\DR}(X_{K},y,x)(A)^{\widehat{\text{cont}}} \rightarrow  \pi_{1}^{\un,\DR}(X_{K},z,x)(A)^{\widehat{\text{cont}}} $$
\noindent and similarly with $d-cont$.

\begin{Definition} i) Let $\pi_{1}^{\un,\DR}(X_{K})^{\widehat{\text{cont}}}$ be the set of all the $\pi_{1}^{\un,\DR}(X_{K},y,x)^{\widehat{\text{cont}}}$ and  $\Pi_{\omega_{\DR}}^{\widehat{\text{cont}}}$ with their topological group structure and the maps above.
\newline We call them the complete (d-)continuous versions of the De Rham fundamental groupoid of $X_{K}$.
\newline ii) Same with $d-cont$ and the terminology d-continuous.
	\end{Definition}

\subsubsection{Ihara action}

\noindent We see that these groupoid keep a notion of Ihara action.

\begin{Proposition} $\pi_{1}^{\un,\DR}(X_{K},1,0)(A)^{\text{cont}}$ and  $\pi_{1}^{\un,\DR}(X_{K},0,0)(A)^{\text{w-cont}}$ are  subgroups for the Ihara product $\circ^{\DR}$ (see \S6 for the Ihara product on $\pi_{1}^{\un,\DR}(X_{K},0,0)$
\end{Proposition}

\begin{proof} As in I-2. \end{proof}

\begin{Remark} This already suggests heuristically that there could exist a "completed" Galois theory of periods of $\pi_{1}^{\un}(X_{K})$, and this is precisely what we are going to concretize in some next parts with a notion of motivic prime weighted multiple harmonic sums.
\end{Remark}

\subsection{The "universal filtered envelope" of $\pi_{1}^{\un,\DR}(X_{K})$}

We have restricted to points of $\Pi_{\omega_{\DR}}$ with values in filtered rings. Let now $A$ be a ring, non-necessarily topological.

\begin{Proposition} \label{la proposition}Let $B$ be a filtered ring as in the previous paragraph, with filtration $\Fil=(B_{s})_{s \in \mathbb{N}}$, and such that we have a morphism $\phi : A \rightarrow B$. Let $\mathcal{L}$ be the group of sequences $(b_{s})_{s\in \mathbb{N}}$ of elements of $B$ such that $b_{s} \in B_{s}$, $b_{s}b_{t} = b_{s+t}$ for all $s,t$, and $b_{s}\phi(A) \subset B_{s}$ for all $s$. One has a map 
	$$ \mathcal{L} \mapsto \text{Hom}_{gp} (\Pi_{\omega_{\DR}}(A), \Pi_{\omega_{\DR}}(B)^{\Fil}) $$
	\noindent defined by sending $(b_{s})_{s \in \mathbb{N}}$ to the map $\big(f = \sum_{w} f[w]w \mapsto \sum_{w} b_{\weight(w)} \phi(f[w])w )$.
\end{Proposition}

\begin{proof} The shuffle equation $f[w]f[w'] = f[w \sh w']$ is sent to $b_{\weight(w)}b_{\weight(w')}f[w]f[w'] = b_{\weight(w\sh w')}f[w \sh w']$, and the result follows by the hypothesis on $b$.
\end{proof}

\begin{Example} \label{l'example}Let $\Lambda$ be a formal variable and let $A[\Lambda]$ be the associated polynomial ring equipped with the $\Lambda$-adic filtration. The sequence of elements $(\Lambda^{s})_{s \in \mathbb{N}}$, is an element of $\mathcal{L}$. Its image by the map of Proposition \ref{la proposition} is $\tau(\Lambda)$, where $\tau$ is the motivic Galois action of $\mathbb{G}_{m}$ expressing the weight grading.
\end{Example}

\begin{Proposition}
	Every element of $ \text{Hom}_{gp}(\Pi_{\omega_{\DR}}(A), \Pi_{\omega_{\DR}}(B)^{\Fil})$ arising from Proposition \ref{la proposition} has a factorization 
	 $\Pi_{\omega_{\DR}}(A) \rightarrow \Pi_{\omega_{\DR}}(A[\Lambda])^{\Fil} \rightarrow \Pi_{\omega_{\DR}}(B)^{\Fil}$ 
	 where the first arrow is the one from Example \ref{l'example}.
\end{Proposition}

\begin{proof} The element of $\text{Hom}_{gp} (\Pi_{\omega_{\DR}}(A), \Pi_{\omega_{\DR}}(A[\Lambda])^{\Fil})$
	\noindent arising from the Proposition \ref{la proposition} in the particular case of Example \ref{l'example} has a right inverse, namely the map 
$\Pi_{\omega_{\DR}}(A[\Lambda])^{\Fil}) \rightarrow \Pi_{\omega_{\DR}}(A)$ defined by the map $A[\Lambda] \rightarrow A$ of moding out by $\Lambda-1$.
	\end{proof}
	
\noindent By an abuse of terminology we could say that $\Pi_{\omega_{\DR}}(A[\Lambda])^{\Fil}$ is the filtered envelope of $\Pi_{\omega_{\DR}}(A)$. Its elements can be viewed as maps $\widehat{\mathcal{O}(\Pi_{\omega_{\DR}})} \rightarrow A[[\Lambda]]$. We will give applications of this object in \S7 (we introduced the formal variable $\Lambda^{\weight}$ in the infinite sums of $\zeta_{p,\alpha}$ in II-1), as well as other applications in IV : it is interesting to see what happens when we replace $\Lambda$ by a complex or $p$-adic variable.

\subsection{Examples of points of $\pi_{1}^{\un}(X_{K})^{\widehat{\text{d-cont}}}$ from the explicit $p$-adic theory of $\pi_{1}^{\un}(\mathbb{P}^{1} - \{0,\mu_{N},\infty\})$}

In this setting, consider $A$ be a topological $K$-algebra.

\subsubsection{Overconvergent $p$-adic hyperlogarithms $\Li_{p,\alpha}^{\dagger}$}

\begin{Corollary} For each $z$, $\Li_{p,\alpha}^{\dagger}(z)$ is a point of $\pi_{1}^{\un,\DR}(X_{K},z,\vec{1}_{0})(K)^{\widehat{cont}}$.
\end{Corollary}

\begin{proof} This follows of the bounds of valuations of Appendix to Theorem I-1 from I-1.
	\end{proof}

\subsubsection{Absolutely convergent infinite sums of $\har_{p^{\alpha}}$ and $\zeta_{p,\alpha}$}

In I-2, where we found the formula for prime weighted multiple harmonic sums :
$$ \har_{p^{\alpha}}(w) = (\Phi_{p,\alpha}^{-1}e_{1}\Phi_{p,\alpha}) \big[\frac{1}{1-e_{0}}e_{1}w \big] $$

\noindent The series in the right-hand side above was absolutely convergent because of bounds of valuations on $\zeta_{p,\alpha}$ from the Appendix to Theorem I-1 from I-1. It was a slight abuse of notation to write $(\Phi_{p,\alpha}^{-1}e_{1}\Phi_{p,\alpha}) \big[\frac{1}{1-e_{0}}e_{1}w \big]$.

\begin{Corollary} $\Phi_{p,\alpha}$ defines an element of $\pi_{1}^{\un,\DR}(X_{K},\vec{1}_{1},\vec{1}_{0})(K)^{\widehat{cont}}$.
\end{Corollary}
	
\begin{Corollary}
	Any absolutely convergent infinite summation of prime weighted multiple harmonic sums, $\sum_{n \geq 0} c_{n} \har_{p^{\alpha}}(w_{n})$, is a coeefficient the element above defined by 
	$\Phi_{p,\alpha}$.
\end{Corollary}

\subsubsection{Remark : algebraic features of the infinite summations $\frac{1}{1-e_{0}}e_{1}w$}

\begin{Definition} Let $\mathcal{O}^{\sh,e_{Z}}_{(1+\lambda e_{0})^{-1}}$ be the localization of the ring $\mathcal{O}^{\sh,e_{Z}}$ equipped with the concatenation product at the multiplicative part generated by elements $1+\lambda e_{0}$, $\lambda \in \mathbb{Q}$. (See \S4.1 for generalities on the localization of non-commutative rings.)
\end{Definition}

\noindent We have a canonical inclusion :
$$ \mathcal{O}^{\sh,e_{Z}}_{(1+\lambda e_{0})^{-1}} \subset \mathcal{O}(\Pi)^{\loc} $$

\begin{Remark} We have a canonical inclusion $$ \mathcal{O}^{\sh,e_{Z}}_{(1+\lambda e_{0})^{-1}} \subset \widehat{\mathcal{O}^{\sh,e_{Z}}}$$ 
	\noindent More generally, we for any multiplicative part $J$ included in $1 + \ker(\epsilon)$, where $\ker(\epsilon)$ is the augmentation ideal, we have a canonical inclusion
	$$ \mathcal{O}^{\sh,e_{Z}}_{J^{-1}} \subset \widehat{\mathcal{O}^{\sh,e_{Z}}}$$
	\noindent where $\mathcal{O}^{\sh,e_{Z}}_{J^{-1}}$ is the localization at $J$.
\end{Remark}

\noindent The product of $\mathcal{O}^{\sh,e_{Z}}$ viewed as $\mathcal{O}(\pi_{1}^{\un,\DR}(X_{K},\omega_{\DR}))$ is the shuffle product $\sh$ (and not the concatenation product), thus it seems necessary to understand whether $\mathcal{O}^{\sh,e_{Z}}_{J^{-1}}$ means something in terms of the ring $\mathcal{O}^{\sh,e_{Z}}$ equipped with $\sh$. The answer is given by :

\begin{Lemma} We also have the concatenation product. We denote by
	Let $f : \mathcal{O}^{\sh,e_{Z}} \rightarrow K$ satisfying $f[w \sh e_{0}] = 0$ for all words $w$. 
	(This is true if $f$ is grouplike and $f[e_{0}]=0$ or if $f$ is primitive). Then we have, for all words $w$, for $i \in \{1,\ldots,N\}$, we have the following equalities in the completion 
	$\widehat{\mathcal{O}^{\sh,e_{Z}}}$. We have 
	$$ f[\frac{1}{1-\lambda e_{0}}e_{z_{i}}w] = f[e_{z_{i}}(\exp_{\sh}(-\lambda e_{0})\text{ }\sh\text{ }w)] $$
	$$ f[we_{z_{i}}\frac{1}{1-\lambda e_{0}}] = f[(\exp_{\sh}(-\lambda e_{0})\text{ }\sh\text{ } w)e_{z_{i}}] $$
\end{Lemma}

\noindent where $\exp_{\sh}(x) = \sum_{l \in \mathbb{N}} \frac{x^{\sh n}}{n!}$.
\newline 
\newline This is a known consequence of the shuffle equation (see for example \cite{IKZ}) and only its interpretation is new.

\subsection{Completed algebraic relations}

As an application, we define a notion essentializing the double shuffle equations between sequences of prime weighted multiple harmonic sums found in II-1 and which involved infinite summations.

\subsubsection{Definitions}

\begin{Definition}
	We call completed algebraic relations between complex, resp. $p$-adic multiple zeta values the vanishing of an absolutely convergent series $\zeta(w)=0$ resp. $\zeta_{p,\alpha}(w)=0$, $w\in \widehat{\mathcal{O}^{\sh,e_{Z}}}$ of $\mathcal{O}_{\sh,e_{Z}}$, such that for all $n \in \mathbb{N}^{\ast}$ if $w_{n}$ is the part of weight $n$ of $w$, we have $\zeta(w_{n})=0$, resp.  $\zeta_{p,\alpha}(w_{n})=0$.
\end{Definition}

\noindent The completed algebraic relations are the absolutely convergent infinite sums of algebraic relations : $\sum_{s \in \mathbb{N}} \zeta(u_{s})= 0$ with, for all $s \in \mathbb{N}$, $u_{s} \in \mathcal{O}^{\sh,e_{Z}}$, $\text{weight}(u_{s})=s$.

\subsubsection{Examples from the explicit $p$-adic theory of $\pi_{1}^{\un}(X_{K})$}

\begin{Example}
	Equations between $\har_{p^{\alpha}}$ from Appendix B of II-1 are completed algebraic relations.
\end{Example}

\noindent This terminology also enables to make a distinction between the equations of II-1 and some other identities involving infinite sums of complex or $p$-adic multiple zeta values. Indeed, a priori, not all equations involving infinite sums of complex or multiple zeta values are infinite sums of algebraic relations. We give two conjectural counter-examples of it, one in the complex setting and one in the $p$-adic setting.

\subsubsection{A conjectural counter-example in the complex setting}

The following equality is classical : for all $s \in \mathbb{N}^{\ast}$, 
\begin{equation} \label{eq: un} 1 = \sum_{l \in \mathbb{N}}  \frac{(s+l-1) \ldots s(s-1)}{(l+1)!} \big(\zeta(s+l) - 1 \big) 
\end{equation}

\begin{Remark} In equation (\ref{eq: un}) there are infinitely many terms of weight $0$. Thus (\ref{eq: un}) is not of the form appearing in the definition above and we think that (\ref{eq: un}) is not equivalent to a completed algebraic relation.
\end{Remark}
\noindent The analogue of  (\ref{eq: un}) in depths $\geq 2$ of (\ref{eq: un})
is due, to our knowledge, to Goncharov and independently to Ecalle : for all $s_{1},\ldots,s_{d} \in \mathbb{N}^{\ast}$ with $s_{d} \geq 2$ :
\begin{equation} \label{eq:GoncharovEcalle} \zeta(s_{d},\ldots,s_{3},s_{1}+s_{2}-1) =
\sum_{l \in \mathbb{N}} \frac{(s_{1}+l-1) \ldots s_{1}(s_{1}-1)}{(l+1)!} \zeta(s_{d},\ldots,s_{2},s_{1}+l) 
\end{equation}

\begin{Remark} Assume that the part of each weight of (\ref{eq:GoncharovEcalle}) vanishes ; this equation would imply the vanishing of $\zeta(s_{d},\ldots,s_{3},s_{1}+s_{2}-1)$, that is the term of weight $(\sum_{i=1}^{d} s_{i}) - 1$, and of each $\frac{(s_{1}+l-1) \ldots s_{1}(s_{1}-1)}{(l+1)!} \zeta(s_{d},\ldots,s_{2},s_{1}+l)$, which is the term of weight $(\sum_{i=1}^{d} s_{i}) + k$, whereas they are all real numbers $>0$. We think that (\ref{eq: un}) is not equivalent to a completed algebraic relation.
\end{Remark}

\subsubsection{A conjectural counter-example in the $p$-adic setting}

In I-2 we proved (Corollary I-2.a, \S1), for all $s_{1},\ldots,s_{d}\in \mathbb{N}^{\ast}$, primes $p$, $\alpha \in \mathbb{N}^{\ast}$
\begin{equation} \label{eq:I2a} \har_{p^{\alpha}}(s_{d},\ldots,s_{1}) = \sum_{l_{1},\ldots,l_{d} \geq 0} \bigg( \sum_{i=0}^{d} \prod_{j=i}^{d} {-s_{i} \choose l_{i}} \bigg) \zeta_{p,\alpha}(s_{i+1}+l_{i+1},\ldots,s_{d}+l_{d}) \zeta_{p,\alpha}(s_{i},\ldots,s_{1}) 
\end{equation}

\noindent To the intepretations of this equality found in II-1 and II-2, we can add to it the following observation.

\begin{Remark}
	\noindent Let us view (\ref{eq:I2a}) as an equality in $\mathbb{Z}_{p}$ for each $p$ separately, instead of in $\prod_{p} \mathbb{Z}_{p}$ as in II-1 and II-2. Assume that the the term of each given weight of this equality vanishes. The part of weight $0$ is the rational number $\har_{p^{\alpha}}(s_{d},\ldots,s_{1})$ ; it is a strictly positive real number as soon as the sum is non-empty, i.e. as soon as $p^{\alpha}>d$, whence a contradiction. In all remaining cases, it would imply the vanishing, for each $l \in \mathbb{N}^{\ast}$, of 
	$$ \sum_{l_{1}+\ldots+l_{d} = l} \bigg( \sum_{i=0}^{d} \prod_{j=i}^{d} {-s_{i} \choose l_{i}} \bigg) \zeta_{p,\alpha}(s_{i+1}+l_{i+1},\ldots,s_{d}+l_{d}) \zeta_{p,\alpha}(s_{i},\ldots,s_{1}) $$
	\noindent We think the vanishing of all these quantities at the same time contradicts the usual conjectures, for example the one on the dimension of $\mathbb{Q}$-vector spaces on $p$-adic multiple zeta values. We thus think that (\ref{eq:I2a}) viewed in each $\mathbb{Z}_{p}$ is never equivalent to a completed algebraic relation.
\end{Remark}

\section{The localization $\pi_{1}^{\un,\DR}(X_{K})^{\loc}$ of $\pi_{1}^{\un,\DR}(X_{K})$}

In this section, we consider the curve $X_{Z} = \mathbb{P}^{1} - Z$, with $Z = \{z_{0},\ldots,z_{r+1}\}$ over $\mathbb{C}$, with $z_{0} = 0$, $z_{r}=1$, $z_{r+1}=\infty$ ($r \in \mathbb{N}^{\ast}$). The following concept essentializes our computations in the "De Rham-rational" setting and will help us to see certain sequences of prime weighted multiple harmonic sums as periods in \S\ref{periods}.

\subsection{Definition}

\subsubsection{Review of generalities on the localization of non-commutative rings}

The notion of localization of non-commutative rings is described as follows. If $R$ is a ring and $S$ is a multiplicative subset such that $(R,S)$ satisfies the so-called Ore's conditions, then, one has a notion of localization $RS^{-1}$ of $R$ with respect to $S$, which can be defined and described in the same way with the commutative case ; Ore's conditions are kind of weak commutation conditions between elements of $R$ and elements of $S$.
\newline When Ore's conditions are not satisfied, the localization of non-commutative rings still exists, however its description is different. The localization of any ring $R$ with respect to any multiplicative subset $S$ is defined as follows. Consider the subfunctor of $\Hom(R,-)$ defined by the homomorphisms mapping $S$ to units. It is continuous and it satisfies the solution set condition, thus, by the representable functor theorem, it is representable.

\begin{Definition} The ring $R S^{-1}$ representing this functor is the localization of $R$ at $S$.
\end{Definition}

\noindent Explicitly, $R S^{-1}$ is the ring whose elements are sums of elements of the form $r_{1}s_{1}^{-1}r_{2}s_{2}^{-1}\ldots r_{i}s_{i}^{-1}$, with $r_{i} \in R$, $s_{i} \in S$.
\newline The most natural examples of localization of non-commutative rings arise usually from rings of differential operators, and this is exactly our situation.

\subsubsection{Review of hyperlogarithms}

We take as usual $\vec{1}_{0}$ as the starting point of paths of integration. Hyperlogarithms are defined as the following (homotopy-invariant) iterated integrals : for any $\gamma$, path on $\mathbb{C} - \{z_{0},z_{1},\ldots,z_{r}\}$ with possibly tangential extremities, starting at the tangential base-point $\vec{1}_{0}$, let the (regularized if needed) iterated integral 

$$ \Li[e_{z_{i_{n}}}\ldots e_{z_{i_{1}}}](\gamma)  = \int_{\gamma} \omega_{z_{i_{n}}} \ldots \omega_{z_{i_{1}}} $$
\noindent For our purposes, we will restrict to paths  $\gamma$  without loops around a singularity, from $\vec{1}_{0}$ to an element $z \in \mathbb{C}$ such that $|z|<1$, and we will denote then by notation is 
$$  \Li[e_{z_{i_{n}}}\ldots e_{z_{i_{1}}}](z) $$
\noindent Because iterated integrals satisfy the shuffle equation, this defines a point $\Li(z) \in \pi_{1}^{\un,\DR}(X_{Z},z,\vec{1}_{0})$ (we have 
$\Li[w]\Li[w'] = \Li[w \text{ }\sh\text{ }w']$).

\subsubsection{Inversion of integration operators and  hyperlogarithms}

Implicitly in the previous paragraph, one has a correspondence between letters of the alphabet $e_{Z}$ and integration operators 
$$ e_{z_{i}} \mapsto \big( f \mapsto \int  f\omega_{z_{i}} \big) $$
	\noindent and a correspondence between words over $e_{Z}$ and integration operators 
$$ e_{z_{i_{d}}}\ldots e_{z_{i_{1}}} \mapsto (\text{composition of the operators above from the right to the left}) $$
\noindent Thus, formally we have 
$$ e_{z_{i}}^{-1} \mapsto \big( f \mapsto (z - z_{i}) \frac{df}{dz} \big) $$

\begin{Definition} For any word over the letters $e_{z_{i}}^{\pm 1}$, let $\Li^{\loc}[e_{z_{i_{n}}}^{\pm 1} \ldots e_{z_{i_{1}}}^{\pm 1}](z)$ be the function of $z$ obtained this way. Let us call these functions localized hyperlogarithms.
\end{Definition}

\noindent One can apply, to the shuffle equation of hyperlogarithms, $\Li[w]\Li[w'] = \Li[w \sh w']$ for all words $w,w'$, the operators $e_{z_{i}}^{\pm 1}$, $i \in \{0,\ldots, r\}$. The operators $e_{z_{i}}^{-1}$ are derivations : $e_{z_{i}}^{-1}(fg) = e_{z_{i}}^{-1}(f)g+ f.e_{z_{i}}^{-1}(g)$. We obtain a generalization of the shuffle equation applying to $\Li^{\loc}$. Let us call it the localized shuffle equation.

\subsubsection{Definition}

\noindent We recall (\S2.2.1) that $\mathbb{O}^{\sh,e_{Z}} = \mathcal{O}(\Pi_{\omega_{\DR}}(X_{Z})$ is the shuffle Hopf algebra over the alphabet $e_{Z} = \{e_{0},e_{z_{1}},\ldots,e_{z_{r}}\}$.

\begin{Definition} i) Let 
$\mathcal{O}(\Pi)^{loc}$ be the localization of $\mathbb{O}^{\sh,e_{Z}}$ for the concatenation cproduct at the multiplicative part generated by the letters of $e_{z}$.
	\end{Definition}
	
\noindent This object can also be denoted by $\mathbb{Q}\langle e_{z_{0}}^{\pm 1},e_{z_{1}}^{\pm 1},\ldots,e_{z_{r}}^{\pm 1} \rangle$ and is simply a free non-commutative Laurent polynomial algebra. Let $B$ a basis of it.

	\begin{Definition} 
		Let $\pi_{1}^{\un,\DR}(X_{K},\omega_{\DR})^{\loc}$ be the affine sub-scheme $\mathbb{A}^{B}$ made of the points whose restriction to words over $e_{Z}$ satisfy the shuffle equation, and who satisfy more generally the "localized shuffle equation" from the previous paragraph.
\end{Definition}

\noindent In particular, we have a map 
$\pi_{1}^{\un,\DR}(X_{K},\omega_{\DR})^{\loc}(\mathbb{C}) \rightarrow \pi_{1}^{\un,\DR}(X_{K},\omega_{\DR})(\mathbb{C})$

\begin{Definition} Let $\pi_{1}^{\un,\DR}(X_{K})^{\loc}$ be the set of schemes $\pi_{1}^{\un,\DR}(X_{K},y,x)^{\loc}$ defined for all couples of base-points $(y,x)$ via $\pi_{1}^{\un,\DR}(X_{K},\omega_{\DR})^{\loc}$ and the canonical isomorphisms of schemes 
	$\pi_{1}^{\un,\DR}(X_{K},\omega_{\DR}) \simeq \pi_{1}^{\un,\DR}(X_{K},y,x)$.
	\newline We call $\pi_{1}^{\un,\DR}(X_{K})^{\loc}$ the localization of the pro-unipotent fundamental groupoid of $X_{K}$.
\end{Definition}
\noindent In particular, since the canonical isomorphisms $\pi_{1}^{\un,\DR}(X_{K},\omega_{\DR}) \simeq \pi_{1}^{\un,\DR}(X_{K},y,x)$ are compatible with the groupoid structure, we have morphisms :	$\pi_{1}^{\un,\DR}(X_{K},z,y)^{\loc} \times \pi_{1}^{\un,\DR}(X_{K},y,x)^{\loc} \rightarrow \pi_{1}^{\un,\DR}(X_{K},z,x)^{\loc}$.
\newline 
\newline 
\noindent Announcements : in a forthcoming work we will study the localization of the pro-unipotent fundamental groupoid for general varieties ; in another forthcoming work we study algebraic functions on $\mathcal{M}_{0,n}$ arising from the localization of $\pi_{1}^{\un,\DR}(\mathcal{M}_{0,n})$, among them some functions obtained from the iterated partial derivatives of multiple polylogarithms, and certain integrals of these algebraic functions. Among several applications, we will show that they keep a form of the generic double shuffle relations, and we will obtain an infinite sequence of subvarieties of $\mathcal{M}_{0,n}$ equipped with algebraic solutions over $\mathbb{Q}$ of the generic double shuffle relations.

\subsection{Multiple harmonic sums with positive and negative indices from I-2 as coefficients of $\Li^{\loc}$}

\noindent In this section we describe the functions $\Li^{\loc}[w]$ for $w$ in the subring $\mathbb{Q} \langle e_{0}^{\pm 1},e_{z_{1}},\ldots,e_{z_{r}} \rangle$.

\begin{Notation} \label{non weighted} We take the following notation for non-weighted multiple harmonic sums :  
	$$ \sigma_{n} \big(
	\begin{array}{c} z_{i_{d+1}},\ldots,z_{i_{1}} \\ s_{d},\ldots,s_{1} \end{array} \big) =  \sum_{0<n_{1}<\ldots<n_{d}<n}
	\frac{\big( \frac{z_{i_{2}}}{z_{i_{1}}} \big)^{n_{1}} \ldots \big(\frac{z_{i_{d+1}}}{z_{i_{d}}}\big)^{n_{d}} 
		\big(\frac{1}{z_{i_{d+1}}}\big)^{n}}{n_{1}^{s_{1}}\ldots n_{d}^{s_{d}}} \in \overline{\mathbb{Q}} $$
\noindent This notation extends to the case where $s_{1},\ldots,s_{d} \in \mathbb{N}^{\ast}$ are replaced more generally by integers $t_{1},\ldots,t_{d} \in \mathbb{Z}$.
\end{Notation}

\noindent These objects appeared as an intermediate object for computing $\circ_{\har}^{\RT,\RT}$ in I-2, \S5, in the rational setting. They can be reduced to $\Li^{\loc}$ (we write it on $\mathbb{P}^{1} - \{0,1,\infty\}$ for simplicity) :

\begin{Corollary} \label{corollary Li loc}We have, for all generalized indices ($t_{d},\ldots,t_{1} \in \mathbb{Z}$, not necessarily positive) :
$$ \Li^{\loc}[e_{0}^{l-1}e_{1}e_{0}^{s_{d}-1}e_{1} \ldots e_{0}^{s_{1}-1} e_{1}](z) =  \sum_{n \in \mathbb{N}} 
\sum_{0<n_{1}<\ldots<n_{d}<n} \frac{1}{n_{1}^{s_{1}}\ldots n_{d}^{s_{d}}}\frac{z^{n}}{n^{l}}$$
$$\Li^{\loc}[e_{0}^{s_{d}-1}e_{1} \ldots e_{0}^{s_{1}-1} e_{1}](z) = \sum_{0<n_{1}<\ldots<n_{d}} \frac{z^{n_{d}}}{n_{1}^{s_{1}}\ldots n_{d}^{s_{d}}}$$
\end{Corollary}

\begin{proof} Same computation with the case where $t_{1},\ldots,t_{d} \in \mathbb{N}^{\ast}$. 
	\end{proof}

\subsection{General application to multiple harmonic sums}

Another feature of $\Li^{\loc}$ is that multiple harmonic sums are no more only coefficients of the series expansions, but values at $z=0$.

\begin{Remark} \label{remark d sur dz}Under the correspondence above between letters of $e_{Z}$ and integration operators, we have $ \frac{d}{dz} = e_{0}^{-1} - e_{1}^{-1}$.	
\end{Remark}

\begin{Corollary} Under the previous definitions, we have, for all $l,t_{d},\ldots,t_{1}\in \mathbb{N}^{\ast}$, and all $n \in \mathbb{N}^{\ast}$
	$$ \Li^{\loc}\bigg[\frac{(e_{0}^{-1}-e_{1}^{-1})^{n}}{n!}e_{0}^{l-1}e_{1}e_{0}^{t_{d}-1}e_{1} \ldots e_{0}^{t_{1}-1} e_{1}\bigg](0) = \frac{1}{n^{l}} 
 \sigma_{n}(t_{d},\ldots,t_{1}) $$ 
	$$ \Li^{\loc}\bigg[\sum_{m=1}^{n}\frac{(e_{0}^{-1}-e_{1}^{-1})^{m}}{m!}e_{0}^{t_{d}-1}e_{1} \ldots e_{0}^{t_{1}-1} e_{1}\bigg](0) = \sigma_{n}(t_{d},\ldots,t_{1})
	$$
	\end{Corollary}

\begin{proof} Follows from Remark \ref{remark d sur dz} and Corollary \ref{corollary Li loc}.
\end{proof}

\section{The rational counterpart at zero $\pi_{1}^{\un,\RT,0}(X_{K})$ of $\pi_{1}^{\un,\DR}(X_{K})$}

The following definitions essentialize our computations in the rational setting. The role usually played by the weight will be played by the depth, and the role played by iterated integrals over a punctured projective line will be played by iterated sums on $\mathbb{N}$.
\newline In the previous papers, such computations started out only as simple analogues to the De Rham computations and ended as having a genuine relation with the De Rham computation. In the same spirit, we state here analogies but try to retrieve De Rham origins to the statements in the rational setting ; the application that we have in mind is in \S7.3 where we will write an analogue of the conjecture of periods of Kontsevich-Zagier type for iterated sums, in our particular cases of multiple harmonic sums in $p$-adic fields.

\subsection{From hyperlogarithms on $\mathbb{A}^{1} - \{z_{0},\ldots,z_{r}\}$ and their series expansions to copies of $\mathbb{N}\text{ }\cup\text{ }\{-1\}$ related by gluing}

Here we consider any curve $\mathbb{A}^{1} - \{z_{0},\ldots,z_{r}\}$ over $\mathbb{C}$ ($z_{0} = 0$, $z_{r}=1$). Let us consider the series expansion $\sum_{n\geq -1} a_{n}(z-z_{i})^{n}$ at $z_{i}$ of functions that are either iterated integrals, or iterated integrals multiplied by a map $f$ such that $f dz$ is one of our differential forms $\omega_{z_{0}},\ldots,\omega_{z_{r}}$. They can be viewed via the maps $n \mapsto a_{n}$, and this point of view can be essentialized as follows.

\begin{Definition} Let $(\mathbb{A}^{1},\mathbb{N})$ be $\mathbb{N}$
	endowed with the $\mathbb{Q}$-vector space of functions $\mathcal{F}_{>0}(\mathbb{N},\overline{\mathbb{Q}})$ having a strictly positive radius of convergence both in the complex and the $p$-adic setting for all primes $p$.
\end{Definition}

\begin{Definition} Let $(\mathbb{G}_{m},\mathbb{N})$ be the set $\mathbb{N} \cup \{-1\}$ endowed with the extended $\mathbb{Q}$-vector space of functions $\mathcal{F}_{>0}(\mathbb{N} \cup \{-1\},\mathbb{Q})$ of functions satisfying the same condition of convergence (the convergence condition does not affect the value at $-1$).
\end{Definition}

\begin{Definition} Let $(\mathbb{A}^{1} - \{z_{0},z_{1},\ldots,z_{r}\},\mathbb{N})$ be the data of $r+1$ copies of $\mathbb{N} \cup \{-1\}$, labeled by $0,z_{1},\ldots,z_{r}$ and transition maps
	$$ \mathcal{F}_{>0}(\mathbb{N} \cup \{-1\},\mathbb{Q})(z_{i}) \rightarrow \mathcal{F}_{>0}(\mathbb{N} \cup \{-1\},\mathbb{Q})(z_{j}) $$
	\noindent such that the transition map from $z_{i}$ to $0$ sends
	$$ \mathds{1}_{n=-1} \text{ at }z_{i}\text{ to }n \mapsto -z_{i}^{-n} \times \mathds{1}_{n\geq 0}\text{ at }0 $$
\end{Definition}

\begin{Remark} Most usual categories of varieties are defined by defining first, the affine varieties with a space of functions, and then the general varieties by gluing the affine varieties together. Here, let us say that $(\mathbb{G}_{m})_{\mathbb{N}}$ plays the role of such an affine variety and that $(\mathbb{A}^{1} - \{z_{0},z_{1},\ldots,z_{r}\})_{\mathbb{N}}$ plays the role of a variety obtained by gluing $r+1$ affine varieties together.
\newline Until now we have made computations in the "De Rham-rational" setting by fixing a base-point as the starting point of paths of integration (by considering the bundle of De Rham paths starting at this point) and considering the coefficients of the series expansions at this base-point. In the present context, this amounts to choose one of the copies of $\mathbb{N} \cup \{0\}$ ; and this is coherent with the analogy above, where each $\mathbb{N} \cup \{0\}$ is viewed as an affine subset : we choose an affine subset in order to make computations in coordinates.
\end{Remark}

\begin{Remark} There is surely a better way to define the gluing above ; the definition that we wrote is only the most primitive possibility. If we embed $\mathbb{N} \cup \{-1\}$ into a vector space of infinite formal sums of elements of $\mathbb{N}$, we avoid problems of convergence and the definition of the transition maps can be made not on the level of $\mathcal{F}_{>0}(\mathbb{N} \cup \{-1\},\mathbb{Q})(z_{j})$ but on the copies of the variant of $\mathbb{N} \cup \{-1\}$. However this is not useful for the present (mostly computational) purposes.
\end{Remark}

\noindent Finally, we have the natural "cohomological equivalence" $\sim$ of elements of $\mathcal{F}_{>0}(\mathbb{N} \cup \{-1\},\mathbb{Q})$ defined by two functions having the same value at $-1$.

\subsection{From multiple polylogarithms on $\mathcal{M}_{0,n}$ and their series expansions to "marked points on $\mathbb{N}$"}

We now consider iterated integrals not on $\mathbb{A}^{1} - \{z_{0},\ldots,z_{r}\}$ but on $\mathcal{M}_{0,n+3}$. 
\newline We consider the simplicial coordinates $(y_{1},\ldots,y_{n})$ on $\mathcal{M}_{0,n+3}$. The series expansions of multiple polylogarithms on $\mathcal{M}_{0,n+3}$ at $(0,\ldots,0)$ can always be written in the form 
\begin{equation} \label{eq:polylog} \sum_{\substack{0<n_{1,1}<\ldots<n_{1,d_{1}}<\\ 
		<\ldots< \\
		n_{u,1}<\ldots<n_{u,d_{u}}<n}} \frac{(\frac{y_{i_{2}}}{y_{i_{1}}})^{n_{1,d_{1}}}}{n_{1,1}^{s_{1,1}}\ldots n_{1,d_{1}}^{s_{1,d_{1}}}} \ldots \frac{(\frac{y_{i_{u}}}{y_{i_{u-1}}})^{n_{u,d}}}{n_{u,1}^{s_{u,1}}\ldots n_{u,d_{u}}^{s_{u,d_{u}}}} 
\frac{(\frac{y_{i_{u+1}}}{y_{i_{u}}})^{n_{u,d_{u}}}}{n^{s}} 
\end{equation}

\noindent with $s_{i,j} \in \mathbb{N}^{\ast}$ and $s \in \mathbb{N}^{\ast}$. This incites us consider now only the highest degree $n$, but to single out the indices $n_{j,d_{j}}$, $j=1\ldots u$. Another motivation for this idea comes from our computations in I-2, \S5. A clearer origin of this idea seems to be from computations on the crystalline fundamental group of $\mathcal{M}_{0,n}$, which will appear in a future work.

\begin{Definition}
	Let $(\mathcal{M}_{0,n+3}), \mathbb{N})$ be the data of $(\mathcal{M}_{0,n+3}), \mathbb{N})$ defined above, plus the data of a map $mar : \{1,\ldots,n\} \rightarrow \mathbb{N}^{\mathbb{N}}$ which we call "point marking function".
\end{Definition}

\noindent The source of the map $mar$ encodes coordinates on $\mathcal{M}_{0,n}$ and image of this map encodes the points $n_{j,d_{j}}$, $j=1\ldots u$ in (\ref{eq:polylog}).

\subsection{From marked points on $\mathbb{N}$ to paths on $\mathbb{N}$ and pro-unipotent paths on $\mathbb{N}$}

The following notions of "paths" and "pro-unipotent paths" on $\mathbb{N}$ (these terminologies being abusive) are motivated by the notion of marked points on $\mathbb{N}$ appearing above. We continue to view $\mathbb{N}$ as included in $\mathbb{N} \cup \{-1\}$, and to view $\mathbb{N} \cup \{-1\}$ as one among several copies of it related by gluing maps as in \S5.1.
For simplicity we restrict to $\mathbb{P}^{1} - \{0,1,\infty\}$.

\subsubsection{Definition}

\begin{Definition} Let $n,m \in \mathbb{N}$ with $n<m$. 
	\newline i) A \emph{path} from $n$ to $m$ is an element of $]n,m[$.
	\newline ii) A \emph{pro-unipotent path} from $n$ to $m$ with $n<m$ is an increasing sequence of elements of $]n,m[$, which we will denote by $n < n_{1} = \ldots = n_{i_{1}} < n_{i_{1}+1} = \ldots = n_{i_{2}} < \ldots 
	<  n_{i_{r-1}+1} = \ldots = n_{i_{r}} < m$.
\end{Definition}

\noindent Let us denote the set of pro-unipotent paths from $n$ to $m$ by $P_{n,m}$. Let the depth of a pro-unipotent path be the number of $<$ minus one. Let $P^{(d)}_{n,m}$ be the part of depth $d$ of $P_{n,m}$ : we have $P_{n,m} = \amalg_{d} P_{n,m}^{(d)}$. Note that $P_{n,m}^{(d)}$ is non empty if and only if $|m-n| > d$.
\newline The terminology "path" is actually motivated by the topology on $\mathbb{N}$ introduced in I-2, \S4.

\subsubsection{Operations on paths}

\begin{Definition} i) Let the associative composition $P_{n,m} \times P_{m,l} \rightarrow P_{n,l}$ be the one given, if $n<m<l \in \mathbb{N}$, by concatenation of sequences :
	$$ (n < n_{1} = \ldots < \ldots  n_{d} < m).(m< m_{1}= \ldots < \ldots m_{d'} < l) $$
	$$ = (n < n_{1} =\ldots < \ldots < n_{d} <m<  m_{1} = \ldots < \ldots m_{d'} < l) $$ 
	\noindent ii) Let the associative and commutative "pre-quasi-shuffle product" 
	$P_{n,m} \times P_{n,m} \rightarrow $(free $\mathbb{Z}$-module over $P_{n,m}$) be given by the logical operation "or" : for example 
	$$ (n<n_{1}<m) \times (n<m_{1}<m) $$
	$$ =  (n<n_{1}<m_{1}<m) + 
	(n<m_{1}<n_{1}<m) + (n<n_{1} = m_{1}<m) $$
	\noindent and where a pro-unipotent path whose depth exceeds $|m-n|$ is send to $0$.
\end{Definition}

\noindent When we will refer to the $\mathbb{Z}[P_{n,m}]$ below, it will mean the free $\mathbb{Z}$-module over $P_{n,m}$, equipped with the pre-series shuffle product, which is a bilineary map $\mathbb{Z}[P_{n,m}] \times \mathbb{Z}[P_{n,m}] \rightarrow \mathbb{Z}[P_{n,m}]$. This is a parallel to the shuffle product which applies to iterated integrals. 
\newline The composition of paths extends uniquely to a bilineary map $\mathbb{Z}[P_{n,m}] \times \mathbb{Z}[P_{m,l}] \rightarrow \mathbb{Z}[P_{n,l}]$. 

\subsubsection{Equivalence classes of paths}

\begin{Definition} Let us say that two pro-unipotent paths in $P_{n,m}$ (having the same number of letters) are equivalent, which we will denote by the symbol $\sim$, if we obtain the same sequence of symbols $=$ and $<$ when removing the letters $n_{i}$.
\end{Definition}

\begin{Proposition}
	i) This is clearly an equivalence relation.
	\newline ii) The set of equivalence classes identifies to the set of words on the alphabet $\{<,=\}$.
\end{Proposition}

\noindent We denote the set of equivalence classes of unipotent paths from $n$ to $m$ by $P_{n,m}/\sim$.

\begin{Proposition} The composition of paths and the pre-quasi-shuffle product induce maps 
	$$ P_{n,m}/\sim \times P_{m,l}/\sim \text{ } \longrightarrow \text{ } P_{n,l}/\sim $$
	$$ \mathbb{Z}[P_{n,m}/\sim] \times \mathbb{Z}[P_{n,m}/\sim]\text{ } \longrightarrow \text{ } \mathbb{Z}[P_{n,m}/\sim] $$
	\noindent obtained by summing on all the paths of an equivalence class. 
\end{Proposition}

\subsection{The iterated summation function on equivalence classes of pro-unipotent paths on $\mathbb{N}$}

The role of homotopy-invariant iterated integrals will be played by functions over equivalence class of paths defined by sums over all the paths of a homotopy class, as for the maps of the previous proposition.
\newline 
\newline Given a path $\gamma=(n_{j})_{j} \in P_{n,m}$ with $n<m$, regroup the indices $n_{i}$ which are equal to each other, and order them according to the order of $\mathbb{N}$. This gives $depth(\gamma)$ equivalence classes which we denote by $1(\gamma) < \ldots < d(\gamma)$. Given an index $j$, let $i=step(j,\gamma)$ be the element of $\{1,\ldots,d\}$ such that $n_{j}$ occurs in $i(\gamma)$. It depends only on the homotopy class of $\gamma$.

\begin{Definition} Let, for each value of the depth $d$, the "summation" coupling :
	\begin{equation} \label{eq:coupling} \big( [\gamma'] \in P_{n,m}^{(d)}, (f_{1},\ldots,f_{d}) \in \mathcal{F}(\mathbb{N},\mathbb{Q})^{d} \big)
	\overset{S}{\longmapsto} \sum_{\substack{\gamma=(n_{i}) \in \\ \text{homotopy class }[\gamma']}}
	\prod_{i} f_{\text{step}(i,[\gamma'])}(n_{i}) 
	\end{equation}
\end{Definition}

\noindent In the case of $f_{1} = \ldots = f_{d}: l \mapsto \frac{1}{l}$, the result of the coupling map is the multiple harmonic sum function and its shifted versions :
$$ H_{n<m} (s_{d},\ldots,s_{1})= \sum_{n<n_{1}<\ldots n_{d}<m} \frac{1}{n_{1}^{s_{1}}\ldots n_{d}^{s_{d}}} $$

\noindent By the formulas for the series expansion of hyperlogarithms, and by identifying the alphabet $\{=,<\}$ to the alphabet $\{e_{0},e_{1}\}$, we can write a relation between this map and the De Rham setting, which we do not write explicitly.

\subsection{From the rational harmonic Ihara action $\circ_{\har}^{\RT,\RT}$ to a "coproduct" on paths}

\subsubsection{A "coaction on paths"}

In this paragraph, let $M$ be a subset of $\mathbb{N}$. Our setting here is all the sets of pro-unipotent paths $P_{n,m}$ with extremities in $M$, i.e. such that $(n,m) \in M^{2}$. We first define a notion of restriction of a path to $M$.

\begin{Definition} Let $\gamma = (n_{j})\in P_{n,m}$, with $(n,m)\in M^{2}$. The restriction of $\gamma$ to $M$, denoted by
	$$ \gamma|_{M} $$
	\noindent is the subsequence of $\gamma$ made of the $n_{j}$'s such that $n_{j} \in M$.
\end{Definition}

\noindent Now we define a notion of homotopy of pro-unipotent paths relative to $M$.

\begin{Definition} We say that two pro-unipotent paths $\gamma_{1},\gamma_{2} \in P_{n,m}$, with $(n,m) \in M^{2}$ are homotopic relative to $M$ if and only if
	\newline i) they are homotopic
	\newline ii) their steps that are in $M$ are the same
\end{Definition}

\begin{Proposition} This defines clearly an equivalence relation.
\end{Proposition}

\noindent We will denote this equivalence relation by $\sim_{M}$. The map $\circ_{\har}^{\RT,\RT}$ factorizes through the following operation.

\begin{Definition}
	Let, for $n,m \in M$, the function
	$$ \Delta_{M} : \mathbb{Z}[P_{n,m}] \mapsto \mathbb{Z}[P_{n,m}] \otimes \big( \otimes_{n<l<l'<m} \mathbb{Z}[P_{l<l'}] \big) $$
	$$ \gamma \mapsto \gamma|_{M} \otimes \big( \otimes_{i=1^{r(p)-1}} \underset{\text{subpath of }\gamma\text{ from } n_{i(p)} \text{ to }n_{i+1(p)}}{\underbrace{n_{i(p)} < \ldots < n_{i+1(p)} )}} $$
	\noindent where
	$$ \gamma|_{M} = 
	(0 < n_{1(p)} = \ldots = n_{1(p)}
	< \ldots
	< n_{r(p)} = \ldots = n_{r(p)}
	< n) $$
\end{Definition}

\noindent In the case of prime multiple harmonic sums, $\Delta_{M}$ is compatible with the coupling.

\begin{Definition} By summing on an equivalence class, we obtain a map which passes to the quotient :
	$$ \Delta_{p^{k}\mathbb{N},\sim} : \mathbb{Z}[P_{n,m}/\sim] \mapsto \mathbb{Z}[P_{n,m}/\sim]\otimes \big( \otimes_{n<l<l'<m} \mathbb{Z}[P_{l<l'}/\sim] \big) $$
\end{Definition}

\subsubsection{$p$-adic aspects}

Now let us view $\mathbb{N}$ included in $\mathbb{Q}_{p}$. We consider the subset $M=p^{\mathbb{Z}}$.
\newline 
The operation $\circ_{\text{har}}^{\RT,\RT}$ factorizes through the operation $\Delta$. This is actually the most natural operation factorizing by $\Delta$. On the other hand, the results of I-2 comparing $\circ_{\text{har}}^{\RT,\RT}$ and $\circ_{\text{har}}^{\DR,\RT}$ show that there is a true relation between $\Delta_{p\mathbb{N}}$ as above and the Goncharov coproduct.

\section{Reformulation of the definitions of the harmonic Ihara actions}

As an application of the previous concepts, we reformulate in an analytic and non $p$-adic way the definition of $\circ_{\har}^{\DR,\RT}$ ; we also give non-$p$-adic definitions for $\circ_{\har}^{\DR,\DR}$ and $\circ_{\har}^{\RT,\RT}$ ; finally we discuss a notion of $\circ_{\har}^{\RT,\DR}$.

\subsection{Preliminaries}

\subsubsection{Motivations}

Although the reason for defining $\circ_{\har}$ concerns the Frobenius (and although we defined $\circ^{\DR,\DR}_{\har}$, $\circ^{\DR,\RT}_{\har}$, $\circ^{\RT,\RT}_{\har}$ with $p$-adic coefficients), the harmonic Ihara action is in itself not subordinated to the Frobenius, its definition and properties make sense independently from it - this is the reason why we denoted $\circ_{\har}$ with $\DR$, and not $\crys$ nor $\rig$ (I-2, I-3). It seems that these feature would look clearer with some more general definitions of $\circ_{\har}$. We will also take this opportunity to reformulate the definition of $\circ_{\har}^{\DR,\RT}$ from I-2 in a more analytic way which makes a little clearer the nature of its connection with what we called the "De Rham-rational setting".
\newline 
\newline However it is subtle to decide whether $\circ_{\har}$ must be thought of as a $p$-adic object or not.
\newline For the moment, most applications concern the $p$-adic setting ; see also \S A.5 for the existence of a natural complex counterpart of the harmonic Ihara action, which suggests that the harmonic Ihara action is primarily a $p$-adic object.
\newline However, in II-2 we saw that the map $\Sigma_{\inv}^{\DR}$, which subjacent to the definition of $\circ^{\DR,\DR}_{\har}$, when defined in the complex setting, had applications to complex multiple zeta values. See also IV-2 for other applications. Most of the algebraic features of the harmonic Ihara action from part II are independent from the chosen base field ; it is not excluded that they should have other applications in the non $p$-adic case.

\subsubsection{Ihara action on the restricted groupoid 	$(\Pi_{y,x})_{y,x \in \{0\} \cup \mu_{N}(K)}$}

We consider the De Rham fundamental groupoid restricted as follows $\Pi^{\text{restr}} = \{ \Pi_{y,x} \text{ | }y,x \in \{0\} \cup \mu_{N}(K) \}$, where $y,x$ are abreviations for $\vec{1}_{y},\vec{1}_{x}$.
\newline 
\newline Let $V_{\omega}$ the pro-affine group scheme defined in \cite{Deligne Goncharov}, \S5.10. It is a quotient of the pro-unipotent motivic Galois group $U^{\omega}$ reviewed in \S\ref{U omega}
\newline By \cite{Deligne Goncharov}, \S5.10, $V_{\omega}$ acts on $\Pi^{\text{restr}}$, by automorphisms ; moreover, let ${}_{1} 1 _{0}$ be the canonical De Rham path on $\Pi_{1,0}$ ; the map $v \mapsto v({}_{1} 1 _{0})$ defines an isomorphism of schemes
\begin{equation} \label{eq:isomorphism V} V \simlra \Pi_{1,0} 
\end{equation}
(\cite{Deligne Goncharov}, Proposition 5.11).
\newline The action of $V^{\omega}$ viewed as its image by the isomorphism (\ref{eq:isomorphism V}) is by the Ihara action. On $\Pi_{0,0}$, this is the map ($z \in \mu_{N}(K)$)
\begin{equation} \label{eq:Ihara product 0} \begin{array}{c}
\Pi_{z,0} \times \Pi_{0,0} \rightarrow \Pi_{0,0} 
\\ (z,f)\mapsto f(e_{0},g_{z_{1}}^{-1}e_{z_{1}}g_{z_{1}},\ldots,g_{z_{N}}^{-1}e_{z_{N}}g_{z_{N}}) 
\end{array} 
\end{equation}
\noindent and, on $\Pi_{z,0}$, $z \in \mu_{N}(K)$, this is the map
\begin{equation} \label{eq:Ihara product z}\begin{array}{c}
\Pi_{z,0} \times \Pi_{z,0} \rightarrow \Pi_{z,0} 
\\ (g_{z},f_{z})\mapsto g_{z}f_{z}(e_{0},g_{z_{1}}^{-1}e_{z_{1}}g_{z_{1}},\ldots,g_{z_{N}}^{-1}e_{z_{N}}g_{z_{N}}) 
\end{array} 
\end{equation}
\noindent where in both cases, $g_{z_{i}} = (x \mapsto z_{i}x)_{\ast}(g_{z_{N}})$ ($g_{z_{N}}=g$).
\newline The map of (\ref{eq:Ihara product z}) is the Ihara product on $\Pi_{z,0}$, and we denote the right-hand side of (\ref{eq:Ihara product z}) by $$ g_{z} \circ^{\DR} f_{z} $$
\noindent 
Until now, we did not have to consider the Ihara action on $\Pi_{0,0}$ ; in order to distinguish it from the other one with the one on $\Pi_{z,0}$, $z \in \mu_{N}(K)$, we denote the right hand of (\ref{eq:Ihara product 0}) by
$$ g \circ^{\DR}_{0,0} f $$

\subsubsection{Ihara action on sections of a bundle of paths starting at $\vec{1}_{0}$}

\noindent We choose any complex or $p$-adic analytic subspace of $\mathbb{P}^{1} /\mathbb{C}$, resp. $\mathbb{P}^{1,an}/K$, where $K=\mathbb{Q}_{p}(\xi)$, $\xi$ a primitive $N$-th root of unity, and any branch of the complex resp. $p$-adic logarithm of $\mathbb{C}$ or $\mathbb{C}_{p}$. 
\newline We consider the bundle of paths of $\pi_{1}^{\un,\DR}(X_{K})$ starting at $\vec{1}_{0}$, and with end-points in this subspace, and $S_{an}$ its group of analytic sections, having at most a logarithmic singularity at $0$. By an analogy with the previous setting, we denote by 

\begin{Definition} Let $V^{an}$ be the group $S^{\an} \times \Pi_{1,0}$.
\end{Definition}

\noindent The motivic Galois action on $\pi_{1}^{\un,\DR}(X_{K})$ at base-points other than the ones of $\Pi^{\restr}$ factorizes through the following map, which we call the Ihara action on $S^{\an}$, or analytic Ihara action (of $S^{\an}$
\begin{equation} \label{eq:Ihara product an} 
\begin{array}{c} (S^{\an} \times \Pi_{1,0}) \times S^{an} \rightarrow S^{an}
\\ 
(g,g_{1}) \circ f = \big( z \mapsto g(z) f(z)(e_{0},g_{z_{1}}^{-1}e_{z_{1}}g_{z_{1}},\ldots,g_{z_{N}}^{-1}e_{z_{N}}g_{z_{N}}) \big)
\end{array}
\end{equation}

\subsubsection{Review of the setting for $\circ_{\har}$ and confrontation with \S3}

\noindent In I-2, I-3, for any $A$ complete topological $K$-algebra, where $K = \mathbb{Q}_{p}(\xi)$, $\xi$ a primitive $N$-th root of unity, the De Rham-rational harmonic Ihara action was defined as a continuous action of the topologial group ($z \in \mu_{N}(K)$) :
$$ \tilde{\Pi}_{z,0}(A)_{\Sigma} = \tilde{\Pi}_{1,0}(A) \cap \Pi_{1,0}(A)_{\Sigma} $$
\noindent equipped with the group law $\circ^{\DR}$, where,
$$ \tilde{\Pi}_{z,0}(A) = \{f \in \Pi_{z,0}(A)\text{ | } f[e_{z}] = f[e_{0}] = 0 \} $$
\noindent 
$$ \Pi_{z,0}(A)_{\Sigma} = \{ f \in \Pi_{z,0}(A)\text{ | for all d }\in \mathbb{N}^{\ast}, \underset{w\text{ word of weight s, depth d}}{\limsup} |f[w]|_{A} \underset{s \rightarrow \infty}{\rightarrow} 0\} $$
\noindent and the topology is induced by the topology of convergence of functions $\{\text{words on }e_{Z}\} \rightarrow A$ which is uniform on each $\{\text{words of depth d}\}$, and the group law is the usual De Rham product $\times^{\DR}$ or the Ihara product $\circ^{\DR}$. The Ihara action was a map 
$$ \tilde{\Pi}_{z,0}(A)_{\Sigma} \times \Map(\mathbb{N},A\langle\langle e_{Z} \rangle\rangle^{\const}) \rightarrow \Map(\mathbb{N},A\langle\langle e_{Z} \rangle\rangle)^{\const}) $$
\noindent where $A\langle\langle e_{Z} \rangle\rangle^{\const}$ was a free submodule of $A\langle\langle e_{Z} \rangle\rangle$.
\newline
\newline By \S3 we will assume that we have chosen an increasing ring filtration $\Fil$ on $\mathbb{Q}$, that we can replace $A$ by a ring with a decreasing separated filtration ${}_{Fil} A$ as in \S3, and the group $\Pi_{z,0}(A)_{\Sigma}$ by 
$$ \pi_{1}^{\un,\DR}(X_{K})(A)^{\widehat{d-cont}} \times \Map(\mathbb{N},{}_{Fil_{A}} A\langle\langle e_{Z} \rangle\rangle)^{\const}) \rightarrow \Map(\mathbb{N},{}_{Fil_{A}} A\langle\langle e_{Z} \rangle\rangle)^{\const}) $$

\noindent Let $A \langle \langle e_{Z} \rangle\rangle^{\const} \subset A \langle\langle e_{Z} \rangle\rangle$ be the vector subspace consisting of the elements $f\in A\langle\langle e_{Z} \rangle\rangle$ such that, for all words $w$ on $e_{Z}$, the sequence $(f[e_{0}^{l}w])_{l\in \mathbb{N}}$ is constant.

\subsection{Analytic and (non-$p$-adic) definition of $\circ_{\har}^{\DR,\RT}$}

We can reformulate in an analytic and non $p$-adic way the definition of $\circ_{\har}^{\DR,\RT}$ ; although this remark is essentially trivial, it seems to clarify the nature of $\circ_{\har}^{\DR,\RT}$.

\begin{Lemma} Let $\log$ be a formal variable. For any group of sections $S$ of the bundles above, We have an injective morphism of groups 
	$$ S \subset \Pi_{\omega_{\DR}}(A[[x]][\log(x)]) $$
\end{Lemma}

\begin{proof} This follows from the existence of the canonical base-point $\omega_{\DR}$ and from that the canonical isomorphisms $\Pi_{\omega_{\DR}} \simeq \Pi_{y,x}$ are compatible with the groupoid structure.	
\end{proof}

\noindent Thus we can work with $\Pi_{\omega_{\DR}}(A[[x]][\log(x)])$, which is canonical, instead of working with an arbitrary group of analytic sections $S_{\an}$. The Ihara actions on all the possibles $S^{\an}$ can be viewed as the map 
$$ G_{an} \times \Pi_{\omega_{\DR}}(A[[x]][\log(x)]) \rightarrow \Pi_{\omega_{\DR}}(A[[x]][\log(x)]) $$ 
\noindent An intermediate object between the map () and the harmonic Ihara action is the term "of degree zero" with respect to one of the factors of $S^{\an}$. 

\begin{Lemma} Under the identification above, the term "of degree zero with respect to $G$", that is the Ihara action of the factor $\Pi_{1,0}$ of of $V_{an}$, is equal to the Ihara action of $\Pi_{1,0}$ on $\Pi_{0,0}(\overline{\mathbb{Q}}[[x]][\log(x)])$.
	\newline Precisely, there is a commutative diagram 
	$$ 
	\begin{array}{ccccc}
	G^{\an} \times \Pi_{\omega_{\DR}}(\overline{\mathbb{Q}}[[x]][\log(x)]) &
	\overset{\circ^{\DR}_{an}}{\longrightarrow}
	& \Pi_{\omega_{\DR}}(\overline{\mathbb{Q}}[[x]][\log(x)]) \\
	\downarrow{\pr_{2} \times can} && \downarrow{can} \\
	\Pi_{1,0} \times \Pi_{0,0}(\overline{\mathbb{Q}}[[x]][\log(x)]) & \overset{\circ^{\DR}_{0,0}}{\longrightarrow} & \Pi_{0,0}(\overline{\mathbb{Q}}[[x]][\log(x)])\\
	\end{array}
	$$
	\noindent $can$ is the canonical isomorphism, and $pr_{2}$ is the projection on the second factor. We also have the same diagram with vertical arrows reversed, $can$ replaced by $can^{-1}$, and 
	$\pr_{2}$ replaced by the inclusion $i : \Pi_{1,0} \rightarrow G^{an}$, $v \mapsto (1,v)$.
\end{Lemma}

\begin{proof} This is almost a byproduct of the definitions and follows from comparing the formulas (\ref{eq:Ihara product 0}) and (\ref{eq:Ihara product an}).
\end{proof}

\noindent If $L$ is any power series $A[[x]][\log(x)]$, with coefficient in any ring, we denote the coeff of $x^{n}\log(x)^{m}$ in $L$ by $L[x^{n}\log(x)^{m}]$. From now on we restrict ourselves to the following subset :

\begin{Definition} Let $\Pi_{0,0}(\overline{\mathbb{Q}}[[x]][\log(x)])^{\text{const}} \subset \Pi_{0,0}(\overline{\mathbb{Q}}[[x]][\log(x)])$ be the subset made of the points $f$ such that, for each word $w$ and for each $n \in \mathbb{N}$, $\tau(n)f[e_{0}^{l}w][x^{n}\log(x)^{0}]$ is independent of $l$.
\end{Definition}

\noindent The generating series of complex hyperlogarithms $\Li$, resp. Furusho's $p$-adic hyperlogarithms $\Li_{p}^{\KZ}$, define a point of $\Pi_{0,0}(\overline{\mathbb{Q}}[[x]][\log(x)])$. We have a map 
$$
\coeff_{0} :
\begin{array}{c}  \Pi_{0,0}(\overline{\mathbb{Q}}[[x]][\log(x)])^{\const} \rightarrow \Map(\mathbb{N},A\langle\langle e_{Z} \rangle\rangle^{\const}
\\ f \mapsto (n \mapsto f[x^{n}][\log(x)^{0}]) \end{array}	$$
\noindent Following I-2, \S3, let $A \langle\langle e_{Z} \rangle\rangle^{\lim} \subset A\langle\langle e_{Z} \rangle\rangle$ be the vector subspace consisting of the elements $f\in A\langle\langle e_{Z} \rangle\rangle$ such that, for all words $w$ on $e_{Z}$, the sequence $(f[e_{0}^{l}w])_{l\in \mathbb{N}}$ has a limit in $A$ when $l \rightarrow \infty$. We have a map "limit" :
\begin{center}
		$\lim : 
		\begin{array}{c} A \langle \langle e_{Z} \rangle\rangle^{\lim} \rightarrow A \langle \langle e_{Z} \rangle\rangle^{\const}
		\\ f= \sum_{w}f[w]w \mapsto \lim f = \sum_{w} \big( \lim_{l\rightarrow \infty} f[e_{0}^{l}w]\big) w
		\end{array}$ \end{center}
\noindent From now on we place ourselves in the continuous setting of \S3 and we take a filtered ring $(A,\Fil_{A})$ as in \S3. The maps $\coeff_{0}$ and $\lim$ induces in particular a map of filtered algebras 

$$ \Pi_{0,0}(\overline{\mathbb{Q}}[[x]][\log(x)])^{\const} \rightarrow \Map(\mathbb{N},A\langle\langle e_{Z} \rangle\rangle^{\const}
\\ f \mapsto (n \mapsto f[x^{n}][\log(x)^{0}]) $$
$$  {}_{\Fil_{A}} A \langle \langle e_{Z} \rangle\rangle^{\lim} \rightarrow {}_{\Fil_{A}} A \langle \langle e_{Z} \rangle\rangle^{\const} $$

\noindent We can restate the definition of the harmonic Ihara action in an abstract way as follows.

\begin{Proposition-Definition} \label{dR-rt harmonic Ihara action} The De Rham-rational harmonic Ihara action of $X_{K}$ is the well-defined map :
	$$ \circ^{\DR,\RT}_{\har} :
\pi_{1}^{\un,\DR}(X_{K})(A)^{\widehat{\text{d-cont}}} \times \Map(\mathbb{N},{}_{\Fil_{A}} A\langle\langle e_{Z} \rangle\rangle)^{\const}) \rightarrow \Map(\mathbb{N},{}_{\Fil_{A}} A\langle\langle e_{Z} \rangle\rangle)^{\const}) $$
	\noindent defined by the equation
	\begin{center} $(g_{z_{1}},\ldots,g_{z_{N}}) \circ_{\har}^{\DR,\RT} (n \mapsto f_{n}) = \big(n \mapsto \lim \big( \tau(n)(g_{z_{1}},\ldots,g_{z_{N}}) \circ^{\DR}_{\Ad} f_{n} \big)\big)$ \end{center}
	\noindent It fits into a commutative diagram involving the the previous variants of the Ihara action and the filtered variants of the maps $\coeff_{0}$ and $\lim$.
\end{Proposition-Definition}

\begin{proof} The proof of the well-definedness is purely algebraic and works in the same way with I-2, \S3.3. The rest is an immediate adaptation of I-2, \S3.3.
\end{proof}

\subsection{Abstract definitions for $\circ_{\har}^{\DR,\DR}$ and $\circ_{\har}^{\RT,\RT}$ from $\pi_{1}^{\un,\DR}(X_{K})^{\widehat{\text{d-cont}}}$}

In the same way with the previous paragraph, the $p$-adic definitions of $\circ_{\har}^{\DR,\DR}$ and $\circ_{\har}^{\RT,\RT}$ that we gave respectively in I-3 and I-2, admit an immediate lift to the abstract setting of $\pi_{1}^{\un,\DR}(X_{K})^{\widehat{d-cont}}$ and $\pi_{1}^{\un,\DR}(X_{K})^{\widehat{d-cont}}$. We leave the explicit definition to the reader.

\subsection{A rational-De Rham action $\circ^{\RT,\DR}_{\har}$}

We defined in part I three versions $\circ_{\har}^{\DR,\DR}$, $\circ_{\har}^{\DR,\RT}$, and $\circ_{\har}^{\RT,\RT}$ of the harmonic Ihara action ; however we did not defined any notion of $\circ_{\har}^{\RT,\DR}$, which would be used to represent an operation of multiple harmonic sums on $p$-adic multiple zeta values. No such operation has appeared in the computations of part I unlike for the operations of the three other types. One can nevertheless force a definition of $\circ_{\har}^{\RT,\DR}$ in the two following equivalent ways :

\begin{Definition} i) Let $\circ_{\har}^{\RT,\DR}$ be the restriction of $\circ_{\har}^{\DR,\DR}$ to the action of elements of the group $\tilde{\Pi}_{1,0}(K)_{\Sigma}$ that are in the image of $\Sigma^{\RT}$.
	\newline ii) Alternatively, consider $\circ_{\har}^{\RT,\DR}$ and view $K\langle \langle e_{Z}\rangle\rangle^{\const,\RT}$ through its isomorphism with $K\langle \langle e_{Z}\rangle\rangle^{\const,\DR}$. 
\end{Definition}

\noindent It seems to us that the only implicit apparition of this object is in I-3, were the examples of $\circ_{\har}^{\DR,\DR}$ that we were interested in were the actions of generating series of $p$-adic multiple zeta values, and thus of course in the image of $\Sigma^{\RT}$. However this definition does not seem necessary to us for establishing our results.

\section{Sequences of multiple harmonic sums, period maps and period conjectures \label{periods}}

In this paragraph, we discuss the nature of several objects that we treated implicitly as periods in the previous parts : we make more concrete the idea that they should be viewed as periods, by defining motivic counterparts to them and stating period conjectures.

\subsection{$\har_{\mathcal{P}^{\mathbb{N}}}$ as periods of the continuous groupoid $\pi_{1}^{\un,\crys}(X_{K})^{\widehat{\text{cont}}}$ ($\DR$ setting)}

\subsubsection{Motivation}

These are our reasons for defining motivic counterparts of sequences of prime weighted multiple harmonic sums.
\newline 
\newline In I-2 we showed formulas exchanging $p$-adic multiple zeta values $\zeta_{p,\alpha}$ and prime multiple harmonic sums $\har_{p^{\alpha}}$ :
\begin{equation} \label{eq:series expansion}\Phi_{p,\alpha} = \Sigma^{\RT}\har_{p^{\alpha}} 
\end{equation}
\begin{equation} \label{eq:inverse series expansion}\har_{p^{\alpha}} = \Sigma^{\DR}_{\inv} \Phi_{p,\alpha}
\end{equation}
\noindent This last equation is more concretely 
\begin{equation} 
\label{eq: formula har mot}
\har_{p^{\alpha}}(\tilde{w}) = \sum_{z \in Z - \{0,\infty\}} z^{-1} \Ad_{\Phi_{p,\alpha}^{(z)}}(e_{z})[\frac{1}{1-e_{0}}w] 
\end{equation}
\noindent where $w = e_{z_{i_{d+1}}}e_{0}^{s_{d}-1}e_{z_{i_{d}}} \ldots e_{0}^{s_{1}-1}e_{z_{i_{1}}}$, $\tilde{w} =\big(
\begin{array}{c} z_{i_{d+1}},\ldots,z_{i_{1}} \\ s_{d},\ldots,s_{1} \end{array} \big)$, and $\Phi_{p,\alpha}^{(z)} = (x \mapsto zx)_{\ast}(\Phi_{p,\alpha})$.
\newline 
\newline In II-1, we defined a counterpart of the double shuffle relations, and of the Kashiwara-Vergne relations for sequences of prime weighted multiple harmonic sums. The variants of the double shuffle relations, for instance, defined a formal scheme $\DMR_{\text{har,prime}}$ over $\mathbb{Z}[[\Lambda]]$ where $\Lambda$ is a formal variable. We showed that the map $\Sigma_{\inv}^{\DR}$ appearing in equation (\ref{eq:inverse series expansion}) defined a morphism of formal schemes $\DMR_{0} \rightarrow \DMR_{\text{har,prime}}$ where $\DMR_{0}$ is Racinet's pro-affine double shuffle scheme viewed as a formal scheme. The proofs implied that the following numbers were a point of $\DMR_{\text{har,prime}}$,
$$ \sum_{z \in Z - \{0,\infty\}} z^{-1} \Ad_{\Phi_{p,\alpha}^{(z)}}(e_{z})[\frac{1}{1-\Lambda e_{0}}w]  \in \overline{\mathbb{Q}_{p}}[[\Lambda]] $$
\noindent as well as their complex analogues modulo $\zeta(2)$.
\newline In II-2, we showed that, conversely, the map $\Sigma^{\RT}$ appearing in equation (\ref{eq:series expansion}) preserved the series shuffle relation, and sent certain parts of the integral shuffle relation of $\DMR_{\text{har,prime}}$ to parts of the integral shuffle relation of $\DMR_{0}$.
\newline 
\newline A conjecture states that double shuffle, associator and Kashiwara-Vergne relations generate, each of them separately, all polynomial relations over $\mathbb{Q}$ between multiple zeta values. To our knowledge, for cyclotomic multiple zeta values, this is also true for the double shuffle relations.
\newline Finally, it is known (\cite{Souderes}) that the double shuffle relations are true for motivic multiple zeta values (whose definition is recalled in \S2.2). This is also true for associator relations, and thus for Kashiwara-Vergne relations.

\subsubsection{Motivic lifts of $\har_{\mathcal{P}^{\mathbb{N}}}$}

We are led to the following notion of motivic prime weighted multiple harmonic sums. The case for which this definition is the most justified is $\mathbb{P}^{1} - \{0,1,\infty\}$, and the definition looks almost equally justified for $\mathbb{P}^{1} - \{0,\mu_{N},\infty\}$ ; the situation is more uncertain for a general curve $\mathbb{P}^{1} - \{z_{0},\ldots,z_{r}\}$ over a number field, because we do not have the $p$-adic computations of part I ; however we have double shuffle relations from II-1.
\newline 
\newline We choose an increasing ring filtration $\Fil$ on $\mathbb{Q}$ as in \S3.1.1, such that all the infinite sums of $p$-adic multiple zeta values that we want to consider (that is to say the sums $\sum_{n=0}^{\infty} \zeta_{p,\alpha}(w_{n}) = 0$ with $w_{n}$ a $\Fil^{n}$-linear combination of words of weight $n$), are absolutely convergent. For example (as we are going to see) the trivial filtration defined by $\Fil^{n} = \mathbb{Z}$ for all $n$ works ; one also has the example given in \S3.1.1. The set of filtrations which work is partially ordered by the inclusion, and it contains a maximal element, whose conjectural definition is delayed to part III.

\begin{Notation} Let ${}_{\Fil} \mathcal{Z}^{\mot}$ be the filtered algebra over $(\mathbb{Q},\Fil)$ generated by the $\zeta^{\mot}(w)$, and $\widehat{{}_{\Fil} \mathcal{Z}^{\mot}}$ be its weight-adic completion. Let $\alpha \in \mathbb{N}^{\ast}$ ; let ${}_{\Fil}\mathcal{Z}^{p,\alpha}$ be the $\mathbb{Z}$-algebra generated by the $\zeta_{p,\alpha}(w)$'s, and let 	$\widehat{{}_{\Fil}\mathcal{Z}^{p,\alpha}}$ be its $p$-adic completion.
	\newline Same notation with ${}_{\Fil}\mathcal{Z}$.
\end{Notation}

\begin{Definition} Let $\tilde{w} = \big( \begin{array}{c} z_{i_{d+1}},\ldots,z_{i_{1}}\\ s_{d},\ldots,s_{1} \end{array}\big)$ be an index of multiple harmonic sums of $\mathbb{P}^{1} - Z$ and $w = e_{z_{i_{d+1}}}e_{0}^{s_{d}-1}e_{z_{i_{d}}} \ldots e_{0}^{s_{1}-1}e_{z_{i_{1}}}$.
	\newline Let the motivic weighted prime multiple harmonic sums (or prime multiple harmonic sum motives) be :
	$$ \har_{\mathcal{P}^{\mathbb{N}}}^{\text{mot}}(\tilde{w}) = 
	(-1)^{d} \sum_{z \in Z - \{0,\infty\}} z^{-1} \Ad_{\Phi^{(z),\text{mot}}}(e_{z})[\frac{1}{1-e_{0}}\tilde{w}]  \in \widehat{\mathcal{Z}^{\mot}} $$
\noindent Let	$\widehat{\mathcal{Z}_{\text{prime}}^{\text{mot}}} \subset \widehat{\mathcal{Z}^{\mot}}$ the sub-algebra topologically generated by the elements $\har_{\mathcal{P}^{\mathbb{N}}}^{\text{mot}}(\tilde{w})$.
\end{Definition}

\noindent Since the double shuffle and associator relations are true for motivic multiple zeta values, the De Rham part of the Theorem II-1.a lifts to $\har_{\mathcal{P}^{\mathbb{N}}}^{\mot}$ :

\begin{Proposition} The generating series $\har_{\mathcal{P}^{\mathbb{N}}}^{\text{mot}}$ of numbers $\har_{\mathcal{P}^{\mathbb{N}}}^{\text{mot}}(\tilde{w})$ is a $\widehat{\mathcal{Z}^{\mot}}$-point of the $\mathbb{Z}[[\Lambda]]$-formal scheme $\DMR_{\har,\text{prime}}$ of II-1.
\end{Proposition}

\noindent (By the weight grading of $\zeta^{\mot}$, we have natural inclusions $\widehat{\mathcal{Z}^{\mot}} \subset \mathcal{Z}^{\mot}[\Lambda]$ and  $\widehat{\mathcal{Z}^{\mot}} \subset \mathcal{Z}^{\mot}[[\Lambda]]$ : $\zeta^{\mot}(w)\mapsto \Lambda^{\weight(w)}\zeta^{\mot}(w)$.)
\newline 
\newline 
Let us now place ourselves on $\mathbb{P}^{1} - \{0,1,\infty\}$. We show that there exists a well defined map sending  $\har_{\mathcal{P}^{\mathbb{N}}}^{\text{mot}}$ to  $\har_{\mathcal{P}^{\mathbb{N}}}$, by using formula 
(\ref{eq: formula har mot}).

\subsubsection{Period maps}
	
\noindent The next facts are conditional to the existence of the crystalline realization functor reviewed in \S\ref{Yamashita}. The period map $\per_{p,\alpha} : \mathcal{Z}^{\mot} \mapsto \mathcal{Z}_{p,\alpha}$ is a morphism of $\mathbb{Q}$-algebras, and it restricts in particular to a morphism of filtered algebras ${}_{\Fil} \mathcal{Z}^{\mot} \mapsto {}_{\Fil}\mathcal{Z}_{p}$. Let $\mathcal{P}_{\geq 3}$ be the set of prime numbers $\geq 3$.

\begin{Notation} For any subset $I$ of $\mathcal{P}_{\geq 3} \times \mathbb{N}^{\ast}$,
let $\Diag({}_{\Fil} \mathcal{Z}_{I}) \subset \prod_{(p,\alpha) \in I} {}_{\Fil} \mathcal{Z}_{p,\alpha}$ be the sub-algebra generated by diagonal elements $(\zeta_{p,\alpha}(w))_{(p,\alpha)\in I}$. Let us equip $\prod_{(p,\alpha) \in I} \mathbb{Q}_{p}$ with the uniform topology associated with the $p$-adic topologies on $\mathbb{Q}_{p}$ ; let 
$\widehat{\Diag}({}_{\Fil} \mathcal{Z}_{I})$ be the completion of the topological ring $\Diag({}_{\Fil} \mathcal{Z}_{I})$.
\end{Notation}

\begin{Proposition} For simplicity we take the trivial filtration $\mathbb{Z}$ ($\Fil^{n} = \mathbb{Z}$) and we leave to the reader to maximise the result. There exist a continuous surjective map 
\begin{equation} \label{eq:periode completee} \widehat{{}_{\Fil}\mathcal{Z}^{\mot}} \rightarrow	\widehat{\Diag}({}_{\Fil} \mathcal{Z}_{\mathcal{P}\times \mathbb{N}^{\ast}})
	\end{equation}
\noindent sending
	$$ \har_{\mathcal{P}^{\mathbb{N}}}^{\mot}(w) \mapsto \har_{\mathcal{P}^{\mathbb{N}}}(w) $$
\end{Proposition}

\begin{proof} By formula (\ref{eq: formula har mot}) we only have to check the continuity of ${}_{\Fil}\per_{p,\alpha}$ uniformly with respect to $(p,\alpha)$ and the result will the completion of the map ${}_{\mathbb{Z}}\per_{I} = ({}_{\Fil}\per_{(p,\alpha)})_{(p,\alpha)\in I}$. 
By a result of Akagi-Hirose-Yasuda (unpublished) and Chatzitamatiou, we have, for all words $w$,
\begin{equation} \label{eq:minoration} \zeta_{p,\alpha}(w) \in \sum_{s \geq \weight(w)} \frac{p^{s}}{s!}\mathbb{Z}_{p} 
\end{equation}
\noindent which implies $|\zeta_{p,\alpha}(w)|_{p} \rightarrow 0$ when $\weight(w) \rightarrow \infty$, uniformly with respect to $(p,\alpha)$ because $v_{p}(\frac{p^{s}}{s!}) \geq \frac{s}{2}$ for $p \not= 2$.
\end{proof}
	
	\noindent Composing the map (\ref{eq:periode completee}) with, for each $p \in \mathcal{P}$ $\alpha \in \mathbb{N}^{\ast}$, the projection
	 $\widehat{\Diag}({}_{\Fil} \mathcal{Z}_{\mathcal{P}\times \mathbb{N}^{\ast}}) \rightarrow \widehat{\Diag}({}_{\Fil} \mathcal{Z}_{\{p\} \times \mathbb{N}^{\ast}})$ and
	 $\widehat{\Diag}({}_{\Fil} \mathcal{Z}_{\mathcal{P}\times \mathbb{N}^{\ast}}) \rightarrow \widehat{\Diag}({}_{\mathbb{Z}} \mathcal{Z}_{\mathcal{P}\times \{\alpha\}})$, we obtain two other surjective maps from $\widehat{{}_{\mathbb{Z}} \mathcal{Z}^{\mot}}$, sending respectively $\har_{\mathcal{P}^{\mathbb{N}}}^{\mot}$ to 
	$ \har_{\mathcal{P}^{\mathbb{N}}}(w)$,  $\har_{p^{\mathbb{N}}}(w)$, $\har_{\mathcal{P}^{\alpha}}(w)$. For each $\alpha$, $\har_{\mathcal{P}^{\alpha}}(w)$ is a natural lift of Kaneko-Zagier's finite multiple zeta values, and for each $p$, $\har_{p^{\mathbb{N}}}(w)$ is the natural explicit substitute to $p$-adic multiple zeta values and we studied their algebraic properties in II-1 and II-2.

\subsubsection{Period conjecture}

\begin{Lemma}
The map above factorizes through the algebra $\Diag_{I}\widehat{{}_{\Fil}\mathcal{Z}_{p,\alpha}}$ equipped with the weight filtration, for $I$ of the three types above.
\end{Lemma}
\begin{proof} Follows from that the valuation of $\har_{p^{\alpha}}(w)$ is $\geq$ to $\weight(w)$.
	\end{proof}

	\begin{Conjecture} The map between algebras filtered by the weight which factorizes this map is an isomorphism of filtered algebras, with $I$ of the type $\mathcal{P}_{\geq 3} \times \{\alpha\}$, $\{p\}\times \mathbb{N}^{\ast}$, and $\mathcal{P}_{\geq 3} \times \mathbb{N}^{\ast}$.
		\end{Conjecture}
	
	\noindent The conjecture above just says that any absolutely convergent identity between prime weighted multiple harmonic sums $\har_{p^{\alpha}}$ which is true for all $p$ or for all $\alpha$ (we could weaken the assumption into "for an infinite number of $\alpha$") is true in the motivic setting. 
	\newline The filtered algebra $\widehat{{}_{\Fil}\mathcal{Z}^{\mot}}$ is isomorphic to its associated graded. Thus the conjecture implies in particular that the target of this map is also isomorphic to its associated graded. This is the analogue of the conjecture that algebraic relations of multiple zeta values are homogeneous for the weight. 
	
	\begin{Definition} We say that $ \har_{\mathcal{P}^{\mathbb{N}}}(w)$,  $\har_{p^{\mathbb{N}}}(w)$, $\har_{\mathcal{P}^{\alpha}}(w)$ are periods of the continuous completion of $\pi_{1}^{\un,\crys}(\mathbb{P}^{1} - \{0,\mu_{N},\infty\})$.
	\end{Definition}
	
\subsubsection{Application of Goncharov's coaction \label{coaction appl}}

Since the beginning of this work, the substitute to motivic Galois actions that we used was the Ihara action, and the harmonic Ihara action related to it that we defined in I-2. Let us see how the motivic Galois coaction of Goncharov (\S2.2) can be applied to $\har_{\mathcal{P}^{\mathbb{N}}}^{\mot}$. It will be convenient to replace $(\Phi_{\text{KZ}}^{-1}e_{1} \Phi_{\text{KZ}})^{\text{mot}}$ by its exponential version :
$$ (\Phi_{\text{KZ}}^{-1}\exp(2i\pi e_{1}) \Phi_{\text{KZ}})^{\text{mot}} $$
\noindent We can retrieve $\Phi_{\text{KZ}}^{-1}e_{1}\Phi_{\text{KZ}}$ from $\Phi_{\text{KZ}}^{-1}e^{2i\pi e_{z}} \Phi_{\text{KZ}}$ either by 
$$ 2i\pi \Phi_{\KZ}^{-1}e_{1}\Phi_{\KZ} = \log ( \Phi_{\KZ}^{-1}e^{2i\pi e_{1}}\Phi_{\KZ} ) $$
\noindent or by	$$ \big( \Phi_{\KZ}^{-1} \exp(2i \pi e_{z}) \Phi_{\KZ} \big)^{\text{mot}}
\equiv (2i\pi)^{\text{mot}} \big( \Phi_{\KZ}^{-1} e_{1} \Phi_{\KZ} \big)^{\text{mot}} \mod ((2i\pi)^{2})^{\text{mot}} $$
\noindent The advantage of $\Phi_{\KZ}^{-1} \exp(2i \pi e_{z}) \Phi_{\KZ} \in \Pi_{0,0}(\mathbb{C})$ compared to $\Phi_{\KZ}^{-1} e_{1}\Phi_{\KZ}$ is that it is the regularized integration of $\nabla_{\KZ}$ along the path $\gamma^{-1}_{z} \circ c_{z} \circ \gamma_{z}$ defined as the conjugation on a positively oriented simple loop around $z$, based at $-\vec{1}_{z}$, by the usual path $\gamma_{z}$ from $\vec{1}_{0}$ to $-\vec{1}_{z}$. 
\newline One has the following formula of composition of paths for iterated integrals. Let two composable paths $\gamma_{1} \in P^{\text{top}}(X_{Z},b,a)$, $\gamma_{2} \in P^{\text{top}}(X,c,b)$, where $P^{\text{top}}$ denotes the set of topological paths extended to tangential base-points of $X_{Z}$. We have
\begin{equation} I_{\gamma_{2} \circ \gamma_{1}}(c;a_{n},\ldots,a_{1};a) = 
\sum_{k=0}^{n} I_{\gamma_{2}}(c;a_{n},\ldots,a_{k+1};b) I_{\gamma_{1}}(b;a_{k},\ldots,a_{1};a) 
\end{equation}

\noindent We have clearly :

\begin{Lemma} \label{7.4}
i)	Goncharov's coaction formula is valid for iterated integrals on the loop $c$.
	\newline ii) Goncharov's coaction formula commutes to the composition of paths.
	\end{Lemma}

\noindent As a consequence of Lemma \ref{7.4} :

\begin{Lemma} \label{Goncharov formula} The formula for the coproduct (reviewed in \S2.2) 
	\begin{multline} \label{eq: prop Goncharov formula} \Delta \tilde{I} (a_{n+1};a_{n},\ldots,a_{1};a_{0})
	\\ = \sum_{0=i_{0}<i_{1} < \ldots < i_{k} < i_{k+1} = n} 
	\tilde{I}(a_{n+1};a_{i_{k}},\ldots,a_{i_{1}};a_{0}) 
	\otimes 
	\prod_{l=0}^{k} \tilde{I}(a_{i_{l+1}};a_{i_{l+1}-1}\ldots a_{i_{l}+1};a_{i_{l}})
	\end{multline}
	\noindent is also true for iterated integrals on $\gamma^{-1} \circ c \circ \gamma$.
\end{Lemma}

\noindent As a consequence, we will replace $(\Phi_{\KZ}^{-1}e_{1} \Phi_{\KZ})\bigg[\frac{1}{1 - \Lambda e_{0}} w \bigg]$ by
\begin{equation} \label{eq:sept un huit}(\Phi_{\KZ}^{-1}e^{2i\pi e_{1}} \Phi_{\KZ})\bigg[ \frac{1}{1 - \Lambda e_{0}} w \bigg]
\end{equation}

\begin{Remark} Usually, when motivic multiple zeta values are dealt with using the motivic Galois coaction, the set of paths that is involved is made of the straight path $dch$ from $0$ to $1$ and its inverse.
	\newline Here, we are going to use an extended set of paths including as well a loop around one.
	We have to be careful that some of the standard properties of motivic multiple zeta values commonly used are no longer true in this setting : the property $\I(a_{0};a_{1},\ldots,a_{n},a_{n+1}) = 0$ if $a_{1} = \ldots = a_{n}$ (\cite{Brown mixed Tate} \S2.4, I0) is not true if the path of integration is the loop around $1$.Instead, we have the weaker statement :\label{lemma vanish} If $a_{1} = \ldots = a_{n}=0$, then $\tilde{I}_{\gamma_{z}^{-1} \circ c_{z} \circ \gamma_{z}}(a_{0};a_{1},\ldots,a_{n};a_{n+1}) = 0$
\end{Remark}

\noindent In order to confront (\ref{eq: prop Goncharov formula}) and (\ref{eq:sept un huit}), it is convenient to view $\Delta$, by the following lift : identify 
$(u_{n+1};u_{n},\ldots,u_{1};u_{0})$ with the word $e_{u_{n}} \ldots e_{u_{1}}$ of 
$\mathcal{O}(\pi_{1}^{\un,\DR}(X_{Z},u_{n+1},u_{0})) \simeq \mathcal{O}^{\sh,e_{Z}}$ and let :
$$ \Delta^{\DR} : \mathcal{O}(\pi_{1}^{\un,\DR}(X_{Z},a_{n+1},a_{0}))
\rightarrow \mathcal{O}(\pi_{1}^{\un,\DR}(X_{Z},a_{n+1},a_{0})) \otimes 
\otimes_{i<j} \mathcal{O}(\pi_{1}^{\un,\DR}(X_{Z},a_{j},a_{i})) $$
\begin{multline} \label{eq: prop Goncharov formula} 
(a_{n+1};a_{n},\ldots,a_{1};a_{0}) \mapsto 
\\  \sum_{0=i_{0}<i_{1} < \ldots < i_{k} < i_{k+1} = n} 
(a_{n+1};a_{i_{k}},\ldots,a_{i_{1}};a_{0}) 
\otimes 
\otimes_{l=0}^{k} (a_{i_{l+1}};a_{i_{l+1}-1}\ldots a_{i_{l}+1};a_{i_{l}})
\end{multline}
consider a variant of our map of reindexation of the differential forms 
$$ \Sigma_{\inv}^{\DR} : 
\begin{array}{l} \mathcal{O}^{\sh,e_{Z}} \mapsto \widehat{\mathcal{O}^{\sh,e_{Z}}}
\\ w \mapsto \frac{1}{1-e_{0}}w 
\end{array} $$
\noindent Consider now the tensor algebra $T(\mathcal{O}^{\sh,e_{Z}}) = \oplus_{n\geq 0}{\mathcal{O}^{\sh,e_{Z}}}^{\otimes n}$, and denote by 
$$ {\Sigma_{\inv}^{\DR}}^{\otimes} = 
\oplus_{n\geq 0} \big( \Sigma_{\inv}^{\DR} \otimes \id^{\otimes n-1} \big) : T(\mathcal{O}^{\sh,e_{Z}}) \rightarrow \oplus_{n\geq 0}{\mathcal{O}^{\sh,e_{Z}}}^{\otimes n-1} \otimes \widehat{\mathcal{O}^{\sh,e_{Z}}} $$
\noindent i.e. the map which applies ${\Sigma_{\inv}^{\DR}}^{\ast}$ only to the first tensor component. We will denote also $\sh : T(\mathcal{O}^{\sh,e_{Z}}) \rightarrow \oplus_{n\geq 0}{\mathcal{O}^{\sh,e_{Z}}}^{\otimes n-1}$ the shuffling of all tensor components.
\newline Finally, denote by $\tilde{I}_{\gamma_{z}^{-1} \circ c_{z} \circ \gamma_{z}}$ the Hodge-Tate structures arising from \cite{Goncharov} lifting the iterated integrals on $\gamma_{z}^{-1} \circ c_{z} \circ \gamma_{z}$.

\begin{Proposition} The coaction on $(\Phi_{0z}^{-1}e^{2i\pi e_{z}} \Phi_{0z})\bigg[ \frac{1}{1 - \Lambda e_{0}} w \bigg] $ is given by
	$$ \tilde{I}_{\gamma_{z}^{-1} \circ c_{z} \circ \gamma_{z}} \circ \Delta^{\DR} \circ \Sigma_{\inv}^{\DR} = \tilde{I}_{\gamma_{z}^{-1} \circ c_{z} \circ \gamma_{z}} \circ \sh \circ {\Sigma_{\inv}^{\DR}}^{\otimes} \circ \Delta^{\DR} $$
\end{Proposition}

\begin{proof} We compute $\Delta^{\DR} \circ \Sigma_{\inv}^{\DR}$ by injecting the formula defining $\Sigma_{\inv}^{\DR}$ into the formula (\ref{eq: prop Goncharov formula}) for $\Delta$. We see that, by Lemma \ref{lemma vanish}, in the formula (\ref{eq: prop Goncharov formula}), the factor $l=k$ on the right $I(a_{n+1};\ldots,a_{i_{k}})$ must contain at least one $a_{i} \not= 0$ in order to be non zero. Then, a simple reindexation of $\Delta^{\DR} \circ \Sigma_{\inv}^{\DR}$ gives that it is equal to $\sh \circ {\Sigma_{\inv}^{\DR}}^{\otimes} \circ \Delta^{\DR}$ Then, by proposition \ref{Goncharov formula}, we can apply Goncharov's formula to $\tilde{I}_{\gamma_{z}^{-1} \circ c_{z} \circ \gamma_{z}}$.
	This gives the result.
\end{proof}

\subsubsection{A period map from the $\DR-RT$ setting}

\subsection{$\har_{\mathcal{P}^{\mathbb{N}}}$ as periods of the localization $\pi_{1}^{\un,\crys}(X_{K})^{\loc}$ / as Taylor periods ($\DR-\RT$ setting)}

\subsubsection{Conjecture}

Since multiple polylogarithms are already equipped with a period map, we only have to defined a map from the algebra co hyperlogarithms to multiple harmonic sums. We construct such a map called $\text{Tay}_{\text{p-adic}}$ in II-2.

\begin{Conjecture} For any absolutely relation among $\har_{p^{\alpha}}$ valid either for all $p$ which does not divide $N$, or for infinitely many $\alpha$, it is implied by a relation between multiple polylogarithms arising from the quotient above.	
\end{Conjecture} 

\subsubsection{Terminology}

In order to fix the ideas, let us say that sequences of multiple harmonic sums are "Taylor periods".

\subsection{$\har_{\mathcal{P}^{\mathbb{N}}}$ as periods via a conjecture of Kontsevich-Zagier type for series ($\RT$ setting)}

\subsubsection{Motivation}

In this paragraph, we try to essentialize the numerous analogies and connections between the proofs in the "De Rham" setting and the proofs in what we call the "rational" setting. What we called the rational setting was the following.
\newline 
\newline In I-2, we studied the operation of passing from $\har_{n}$ to $\har_{p^{\alpha}}n$ by using only elementary operations on multiple harmonic sums, obtained a result purely in terms of multiple harmonic sums, and showed that the result factorized through an operation arising $\pi_{1}^{\un,\crys}(\mathbb{P}^{1} - \{0,\mu_{N},\infty\})$.
\newline 
\newline In I-3, we studied prime multiple harmonic sums $\har_{p^{\alpha}}$ as functions of $\alpha$, again by using only elementary operations, and the same phenomena happen.
\newline 
\newline In II-1, we recalled why the double shuffle as well as other standard families of operations on multiple harmonic sums reflecting relations between hyperlogarithms could be obtained by using only elementary operations on multiple harmonic sums.
\newline 
\newline In order to essentialize these phenomena, we will formulate a conjecture by imitating Kontsevich-Zagier's period conjecture.

\subsubsection{Preliminary : Kontsevich-Zagier's period conjecture}

Kontsevich and Zagier gave in their paper \cite{Kontsevich Zagier} an elementary definition of periods : "a period is a complex number whose real and imaginary parts are values of absolutely convergent iterated integrals of rational functions with rational coefficients, over domains in $\mathbb{R}^{n}$ given by polynomial inequalities with rational coefficients." (\cite{Kontsevich Zagier}, \S1.1).
\newline They also formulated a variant of Grothendieck's period conjecture.

\begin{Conjecture} (Kontsevich Zagier, \S1.2, conjecture 1). If a period has two different integral expressions, we can pass from one to another by using the following rules :
\newline - additivity of the integral with respect to the integrand and the domain of integration.
\newline - invertible change of variable
\newline - the "Newton-Leibniz formula", i.e. the expression of the difference of a value of a function $f$ at two points by the integral of the derivative of $f$.
\end{Conjecture}

\noindent In short, all algebraic relations between periods should be reflected by operations on integrals.

\subsubsection{A list of operations on multiple harmonic sums and an analogue of Kontsevich-Zagier's conjecture for $p$-adic relations of multiple harmonic sums}

Motivated by the facts recalled above from the previous papers, we now write a variant of Kontsevich-Zagier's conjecture by replacing iterated integrals by iterated series in the sense of \S5 in our particular case of multiple harmonic sums.
\newline More precisely, in the analogy provided by \S5, the duality between the domain of integration and the integrand for integrals is replaced by the duality between the domain of summation (equivalence classes of "pro-unipotent paths on $\mathbb{N}$) and the summand (certain elementary functions on $\mathbb{N}$). The other features of our setting is that we authorize ourselves to use the structure of topological field of $\mathbb{C}_{p}$, and to write things such $\frac{1}{1-x} = \sum_{n \in \mathbb{N}} x^{n}$ for $x \in \mathbb{C}_{p}$, $|x|_{p} < 1$.
\newline We place ourselves on $\mathbb{P}^{1} - \{0,\mu_{N},\infty\}$.

\begin{Definition} \label{functions}Consider the functions $\amalg_{d \in \mathbb{N}}(\mathbb{N}^{\ast})^{d} \rightarrow \mathbb{C}_{p}$ and on subsets on $\amalg_{d \in \mathbb{N}}(\mathbb{N}^{\ast})^{d}$ made out of the following functions by taking sums, products, inverses, and composition : 
	\newline i) rational functions on $\mathbb{N}$ minus certain points 
	\newline ii) $n \mapsto z_{i}^{n}$, where $z_{i} \in \mu_{N}(\mathbb{C}_{p})$
	\newline iii) $n \mapsto n!$.
	\newline iv) binomial coefficients of affine functions of the coordinates
\newline 
\end{Definition}

\noindent By \S5, the multiple harmonic sums are obtained from the functions i) by applying the iterated summation function (and using $n \mapsto \frac{1}{n^{s}}$).

\begin{Definition} \label{operations}Consider the following operations :
	\newline i) Additivity of the summation with respect to the domain of summation - in the case of iterated sums, this includes the series shuffle product.
	\newline ii) Additivity of the summation with respect to the summand.
	\newline iii) Change of variable.
	\newline iv) Reindexation and cutting of the sums
	\newline v) The structure of topological field of $\mathbb{C}_{P}$ ; this includes taking limits, absolutely convergent sums, and using the expansion $\frac{1}{1-x} = \sum_{n\geq 0} x^{n}$ for $|x|_{p}<1$
\end{Definition}

\begin{Conjecture} If we have an equality between prime multiple harmonic sums $\har_{p^{\alpha}}$, which is valid either for all $p$ prime to $N$, or for all infinitely many $\alpha \in \mathbb{N}$ in $\prod_{p} \mathbb{Q}_{p}$, we can pass from one to another using the operations of Definition \label{operations} on the algebra of functions of Definition \label{functions}, use at the same time for all $p$ prime to $N$, resp. for all $\alpha$.
\end{Conjecture}

\subsection{The couples $(\zeta_{p,\alpha}(w),\har_{p^{\mathbb{N}}}(w))$ and "explicit $p$-adic periods"\label{Galois theory}}

\subsubsection{Motivation}

The Galois theory of multiple zeta values is not exclusively formulated via motives. Since motives are abstract objects, for computational purposes it is in general more convenient to replace them by much more concrete objects, in particular their image in the Hodge realization : since the Hodge realization functor of categories of mixed Tate motives is fully faithful (\cite{Deligne Goncharov}, Proposition 2.14), this is a way to attain certain properties of motivic multiple zeta values. Another example of this phenomenon is provided by the use of the Ihara action in the De Rham realization.
\newline A fortiori for periods for which the motivic aspects are missing, the philosophy of the Galois theory of periods still applies without motives, typically by dealing with two Tannakian categories such as the ones underlying the De Rham and Betti realizations of the $\pi_{1}^{\un}$ and the comparison isomorphism between them (see II-1, \S2 for the definitions in the case of the $\pi_{1}^{\un}$). This gives objects often thought of as substitutes to motivic Galois actions and the motivic lifts of periods ; they are abusively called motivic in several works.
\newline 
\newline Here, we have the motivic multiple zeta values $\zeta^{\mot}$, which lift $\zeta_{p,\alpha}$, and by \S7.1 we also have a notion $\har_{\mathcal{P}^{\mathbb{N}}}^{\mot}$ of motivic weighted prime multiple harmonic sums. However, our strategies of proofs and formulations of results from I-2 to II-2 are not at all summarized by the couple $(\zeta^{\mot},\har_{\mathcal{P}^{\mathbb{N}}}^{\mot})$. We obtained the explicitness of $p$-adic multiple zeta values by using three different computational settings, the $\DR$ (De Rham, or rather $\crys$, $\rig$) setting, the $\RT$ ("rational") setting and the "$\DR-\RT$" setting, where we made as if $\DR$ and $\RT$ refered to two realizations of a same object ; we introduced the harmonic Ihara actions $\circ_{\har}^{\DR,\DR}$, $\circ_{\har}^{\DR,\RT}$, $\circ_{\har}^{\RT,\RT}$, which played the role of motivic Galois actions ; we introduced maps $\Sigma^{\RT}$, $\Sigma^{\DR}_{\inv}$, which played the role of comparison maps.
\newline In this paragraph we want to formalize the analogy with the more general non-motivic framework described above.

\subsubsection{Definition}

\begin{Remark} We will use implicitly the terminology "$p$-adic periods" for $p$-adic multiple zeta values which is slightly abusive ; they are elements of $\mathbb{Q}_{p}$ and not elements of a Fontaine ring.
\end{Remark}

\noindent The object of interest for us is an analogue of certain concrete substitutes to motivic multiple zeta values ; and at the same time it reflects the explicitness of $p$-adic multiple zeta values.

\begin{Definition} \label{sept points dix neuf}i) We call \emph{explicit $p$-adic period} of $\pi_{1}^{un}(\mathbb{P}^{1} - \{0,\mu_{N},\infty\})$ the quadruple $(\zeta_{p,\alpha},\har_{p^{\alpha}},\har_{\mathbb{N}}^{(p^{\alpha})},\har_{p^{\alpha}\mathbb{N}})$ with the equations (\ref{eq:harmonic torsor 1}), (\ref{eq:harmonic torsor 2}), (\ref{eq:exchange}) (\ref{eq:comparison 1}), (\ref{eq:comparison 2}), as well as the generalizations of these objects to multiple harmonic sums with reversals in the sense of II-2.
\newline ii) We also call in the same way the couple $(\zeta_{p,\alpha},\har_{p^{\alpha}})$ plus the equations (\ref{eq:exchange}) (\ref{eq:comparison 1}), (\ref{eq:comparison 2}).
\end{Definition}

\noindent The distinction between i) and ii) corresponds to the two points of view for the explicitness of algebraic relations of $p$-adic multiple zeta values defined in II-2. The primary way to use this notion is the following.

\begin{Problem} Try to lift any algebraic, resp. completed algebraic, property of $p$-adic multiple zeta values, resp. sequences of prime weighted multiple harmonic sums, into a property of "explicit $p$-adic periods".
	\end{Problem}

\noindent This formulation looks vague, but we have precise illustrations of this problem by II-1 and II-2. This problem is a variant of the one "try to lift all properties of multiple zeta values to motivic multiple zeta values".

\begin{Remark} There are several possible refinements and variants of Definition \ref{sept points dix neuf}, the situation being different (and more subtle by certain aspects) than the usual Betti-De Rham one. The Frobenius is primarily an automorphism of the $\pi_{1}^{\un}$ ; one can consider the image by Frobenius of the canonical De Rham path of $\Pi_{1,0}(K)$, i.e. $\Phi_{p,\alpha}$, or the Frobenius-invariant path $\Phi_{p,\infty}$ (to which are associated a variant $\har_{p^{\infty}}$ of prime weighted multiple harmonic sums and other comparison formulas involving $\circ_{\har}^{\DR,\DR}$ and comparison maps ${\Sigma}^{\RT,\infty}$, ${\Sigma}^{\RT,-\infty}$ from I-3) ; one can consider the Frobenius or all its iterates ; the Frobenius or its inverse ; one can consider $\har_{p^{\alpha}}$ as a point of the group $\tilde{\Pi}_{1,0}(K)_{\Sigma}$ which acts by $\circ_{\har}$, or as a restriction of the point $\har_{\mathbb{N}}$ of the harmonic torsor to prime numbers $n=p^{\alpha}$. Etc.
\end{Remark}

\subsection{The sequences $\har_{\mathbb{N}}(w)$ as periods via $\circ_{\har}$ and the prime decomposition on $\mathbb{N}$}

Let us restrict ourselves to $\mathbb{P}^{1} - \{0,1,\infty\}$ for simplicity. In equations (\ref{eq:harmonic torsor 1}), (\ref{eq:harmonic torsor 2}), the point of the harmonic Ihara torsors that is involved is the same for all prime numbers $p$, and equal to $\har_{\mathbb{N}}$, the generating series of all multiple harmonic sums. We want to explain some applications of this remark.

\subsubsection{Formula}

Let us consider the prime decomposition of $n \in \mathbb{N}^{\ast}$ : $n = p_{1}^{\alpha_{1}}\ldots p_{r}^{\alpha_{r}}$ ; by applying several times equation (\ref{eq:harmonic torsor 1}), we obtain :

\begin{Proposition} We have :
	\begin{equation} \label{eq:c product} 
	\har_{n} = \big( \tau(p_{1}^{\alpha_{1}} \ldots p_{r-1}^{\alpha_{r-1}}) \Phi_{p_{r},\alpha_{r}} \circ_{\har}^{\DR,\RT} \ldots  \circ_{\har}^{\DR,\RT} \Phi_{p_{1},\alpha_{1}} \circ_{\har}^{\DR,\RT} 1 \big)(1)
	\end{equation}
\end{Proposition}

\noindent Although it involves $p$-adic numbers with different $p$ at the same time, the formula  (\ref{eq:c product})  makes sense because for each $i$, $\tau(p_{1}^{\alpha_{1}} \ldots p_{i-1}^{\alpha_{i-1}}) (\Phi_{p_{i},\alpha_{i}} \circ_{\har}^{\DR,\RT} \ldots  \circ_{\har}^{\DR,\RT} (\Phi_{p_{1},\alpha_{1}} \circ_{\har}^{\DR,\RT} 1)\ldots)$, $i \in \{1,\ldots,r-1\}$, has coefficients in $\mathbb{Q}$. The formula (\ref{eq:c product}) has been suggested to us by Pierre Cartier a few years ago.

\begin{Remark} This formula looks like an analogue in $\pi_{1}^{\un}(\mathbb{P}^{1} - \{0,1,\infty\})$ (in a very loose sense) of the Eulerian product formula for the values of the Riemann zeta function :
$$ \zeta(s) = \sum_{n\geq 1} \frac{1}{n^{s}} = \prod_{p \in \mathcal{P}} \bigg( 1 - \frac{1}{p^{s}} \bigg)^{-1} $$
\noindent The existence of analogies or connections between multiple zeta functions viewed from the point of view of analytic number theory and $\pi_{1}^{\un}(\mathbb{P}^{1} - \{0,1,\infty\})$, has already been observed in the $p$-adic world \cite{FKMT}.
\end{Remark}

\subsubsection{Application to algebraic relations}

At first sight, since the data of $\har_{\mathbb{N}}(w)$ is essentialy equivalent to the data of the hyperlogarithm $\Li[w]$ on a neighbourhood of $0$, $\har_{\mathbb{N}}(w)$ does not seem to be a new object, unlike $\har_{\mathcal{P}^{\mathbb{N}}}$ treated above. However, the formula (\ref{eq:c product}) reflects certain algebraic properties of $\har_{\mathbb{N}}(w)$.
\newline 
\newline It is standard that the only algebraic relations between hyperlogarithms are the double shuffle relations, and that the integral shuffle relation involves all values $\har_{m}$, $m \in \{1,\ldots,n\}$ at the same time. In particular :

\begin{Fact}
	The $\mathbb{Q}$-algebra of maps $n \mapsto \har_{n}$, $\mathbb{N} \rightarrow \mathbb{C}$, is isomorphic to the quasi-shuffle algebra.
\end{Fact}

\noindent We recall also from II-2 than, on the "harmonic torsors" $\mathcal{T}_{\har}^{\DR,\RT}$ defined in I-2, the harmonic Ihara action preserves the quasi-shuffle relation. As a consequence, one can retrieve from formula (\ref{eq:c product}) the quasi-shuffle relation, and thus by the fact above, all algebraic relations between the maps $n \mapsto \har_{\mathbb{N}}$. This motivates the following definition and terminology.

\subsubsection{Definition}

\begin{Definition}
	Let $\mathcal{T}_{\har,\text{adelic}}$ be the set of points $f = n \mapsto f_{n} \in \Map(\mathbb{N},\mathbb{Q}\langle\langle e_{Z}\rangle\rangle^{\const})$ such that for all $p\in \mathcal{P}$, both $f$ and the map $n \mapsto f_{pn}$ are in the $p$-adic harmonic torsor $\mathcal{T}^{\DR,\RT}_{\har,p} \subset \Map(\mathbb{N},\mathbb{Q}_{p}\langle\langle e_{Z}\rangle\rangle^{\const})$ defined in I-2.
\end{Definition}

\noindent In particular, for each $p \in \mathcal{P}$, there exists an element $\varphi_{p} \in \tilde{\Pi}_{1,0}(\mathbb{Q}_{p})_{\Sigma}$, such that we have 
$$ (n \mapsto f_{pn}) = \varphi_{p} \circ_{\har}^{\DR,\RT}  (n \mapsto f_{n}) $$ 
\noindent and, for all $\alpha \in \mathbb{N}$, if we denote by $\varphi_{p,\alpha}$ the $\alpha$-th Ihara power of $\varphi_{p}$, we have  
$$ (n \mapsto f_{p^{\alpha}n}) = \varphi_{p,\alpha} \circ_{\har}^{\DR,\RT}  (n \mapsto f_{n}) $$ 
\noindent and the map $n \mapsto f_{p^{\alpha}n}$ is in $\mathcal{T}^{\DR,\RT}_{\har,p}$.

\begin{Fact} The set $\mathcal{T}_{\har,\text{adelic}}$ is in bijection with the set
	$G_{\har,\text{adelic}} \subset \prod_{p} \tilde{\Pi}_{1,0}(\mathbb{Q}_{p})_{\Sigma}$ of sequences $(\varphi_{p})_{p \in \mathcal{P}}$ such that, if $\varphi_{p,\alpha}$ is the $\alpha$-th Ihara power of $\varphi_{p}$, $p \in \mathcal{P}$, $\alpha \in \mathbb{N}^{\ast}$, then, 
		for any sequences of prime numbers
		$p_{r},\ldots,p_{1} \in \mathcal{P}$, and any sequences of integers $\alpha_{r},\ldots,\alpha_{1} \in \mathbb{N}^{\ast}$, we have $\varphi_{p_{r},\alpha_{r}} \circ \ldots \circ \varphi_{p_{r},\alpha_{r}} \circ 1$ is in $\mathbb{Q}$ and independent of the order of the $p_{i}$.
\end{Fact}

\begin{Remark} For $\mathbb{P}^{1} - \{0,\mu_{N},\infty\}$ with $N>1$, the situation is similar, although we must restrict to $p \nmid N$ and to the powers $p$ that are powers of $q$.
\end{Remark}

\subsubsection{Terminology}

In order to fix ideas, let us say that $\har_{\mathbb{N}}$ viewed by formula (\ref{eq:c product}), is a "iterated period of the harmonic Ihara action".

\subsection{Finite multiple zeta values  $\zeta_{\mathcal{A}}$ as "finite periods"}

\subsubsection{Finite multiple zeta values and Kaneko-Zagier's conjecture (review)}

Let $\mathcal{P}$ be the set of prime numbers and let the ring of characteristic zero

$$ \mathcal{A} = \big( \prod_{p \in \mathcal{P}} \mathbb{Z}/p\mathbb{Z} \big) / \big( \bigoplus_{p \in \mathcal{P}} \mathbb{Z}/p\mathbb{Z} \big) =\big( \prod_{p \in \mathcal{P}} \mathbb{Z}/p\mathbb{Z} \big) / \big( \prod_{p \in \mathcal{P}} \mathbb{Z}/p\mathbb{Z} \big)_{\text{tors}} \simeq \big( \prod_{p \in \mathcal{P}} \mathbb{Z}/p\mathbb{Z} \big) \otimes_{\mathbb{Z}} \mathbb{Q} $$
\noindent where "tors" refers to the torsion subgroup. 

\begin{Definition} (Zagier) Finite multiple zeta values are the following numbers, for $(s_{d},\ldots,s_{1}) \in \amalg_{d' \geq 1} (\mathbb{N}^{\ast})^{d'}$ :
	\begin{equation} \zeta_{\mathcal{A}}(s_{d},\ldots,s_{1}) = \left(p^{-(s_{d}+\ldots+s_{1})}\har_{p}(s_{d},\ldots,s_{1})\right)_{p \in \mathcal{P}} =  \left(\sum_{0<n_{1}<\ldots<n_{d}<p} \frac{1}{n_{1}^{s_{1}} \ldots n_{d}^{s_{d}}} \right)_{p \in \mathcal{P}} \in \mathcal{A}
	\end{equation}
\end{Definition}
\noindent As usual, by convention, $\zeta_{\mathcal{A}}(\emptyset) = 1$.

\begin{Conjecture} (Kaneko-Zagier)
	For $s \in \mathbb{N}$, let $\mathcal{Z}_{s}$ the $\mathbb{Q}$-vector subspace of $\mathcal{A}$ generated by the finite multiple zeta values of weight $s$. Then, the sum of the $\mathcal{Z}_{s}$'s is direct ; moreover, we have 
	\begin{center} $\sum_{s\geq 0} \dim(\mathcal{Z}_{s})\Lambda^{s} = \frac{1-\Lambda^{2}}{1 - \Lambda^{2} - \Lambda^{3}}$ \end{center}
	\noindent More precisely, the following correspondence defines an isomorphism of $\mathbb{Q}$-algebras from the algebra generated by finite multiple zeta values to the algebra generated by multiple zeta values modulo $(\zeta(2))$ (the right hand side is independent of the chosen regularization) :
	\begin{equation} \label{eq:32} \zeta_{\mathcal{A}}(s_{d},\ldots,s_{1}) \mapsto \sum_{d'=0}^{d}(-1)^{s_{d'+1}+\ldots+s_{d}}\zeta(s_{d'+1},\ldots,s_{d})\zeta(s_{d'},\ldots,s_{1}) \mod \zeta(2)
	\end{equation}
\end{Conjecture}
\noindent This conjecture is surprising ; it is new that multiple zeta values have a conjectural counterpart in the ring $\mathcal{A}$. This conjecture combined to the conjectures of periods for complex and $p$-adic multiple zeta values, this would also imply that (for each $\alpha \in \mathbb{N}^{\ast}$) the maps
$$ \zeta_{\mathcal{A}}(s_{d},\ldots,s_{1}) \mapsto \sum_{d'=0}^{d}(-1)^{s_{d'+1}+\ldots+s_{d}}\zeta_{p,\alpha}(s_{d'+1},\ldots,s_{d})\zeta_{p,\alpha}(s_{d'},\ldots,s_{1}) $$
\noindent are isomorphisms of algebras between the $\mathbb{Q}$-algebra of finite multiple zeta values and the $\mathbb{Q}$-algebra of $p$-adic multiple zeta values $\zeta_{p,\alpha}$, and similarly with the maps 
\begin{equation} \label{eq:finite motivic} \zeta_{\mathcal{A}}(s_{d},\ldots,s_{1}) \mapsto \sum_{d'=0}^{d}(-1)^{s_{d'+1}+\ldots+s_{d}}\zeta^{\text{mot}}(s_{d'+1},\ldots,s_{d})\zeta^{\text{mot}}(s_{d'},\ldots,s_{1})  \mod \zeta^{\text{mot}}(2) 
\end{equation}
\noindent and the $\mathbb{Q}$-algebra of motivic multiple zeta values. Thus the right hand side of (\ref{eq:finite motivic}) can be called motivic finite multiple zeta values and denoted by $\zeta_{\mathcal{A}}^{\text{mot}}$.

\subsubsection{Terminology}

In order to fix ideas, let us say that finite multiple zeta values "are finite periods". This terminology has been suggested to us by Francis Brown a few years ago. 
\newline We do not know whether there is a general notion behind Kaneko-Zagier's conjecture or not (despite that one could speculate such a definition by taking reduction modulo large primes of certain sequences of $p$-adic periods, we do not know any other example justifying such a notion).

\subsubsection{Relation with $p$-adic multiple zeta values (review) and interpretation}

Akagi, Hirose and Yasuda have deduced (unpublished) from our formula for $\har_{p^{\alpha}}$ as an infinite sum of $p$-adic multiple zeta values that :

\begin{Corollary} (Akagi-Hirose-Yasuda) If $p>s_{d}+\ldots+s_{1}$, we have :
	\begin{center} $\sum_{0<n_{1}<\ldots<n_{d}<p} \frac{1}{n_{1}^{s_{1}} \ldots n_{d}^{s_{d}}} \equiv$
	\end{center}
	\begin{center} $ p^{-(s_{d}+\ldots+s_{1})} \sum_{d'=0}^{d}(-1)^{s_{d'+1}+\ldots+s_{d}}\zeta_{p}(s_{d'+1},\ldots,s_{d})\zeta_{p}(s_{d'},\ldots,s_{1}) \mod p$ \end{center}
	\noindent In particular, Kaneko-Zagier's finite multiple zeta values can be expressed entirely in terms of $p$-adic multiple zeta values :
	\begin{center} $\zeta_{\mathcal{A}}(s_{d},\ldots,s_{1})=$ \end{center}
	\begin{center} $ \big( p^{-(s_{d}+\ldots+s_{1})} \sum_{d'=0}^{d}(-1)^{s_{d'+1}+\ldots+s_{d}}\zeta_{p}(s_{d'+1},\ldots,s_{d})\zeta_{p}(s_{d'},\ldots,s_{1}) \mod p \big)_{p>s_{d}+\ldots+s_{1}} \in \mathcal{A}$ \end{center}
\end{Corollary}

\noindent We have called it in II-1 the DR expression of $\zeta_{\mathcal{A}}$, whereas the original one is the $\DR-\RT$ or $\RT$ expression. This explains the formula appearing in Kaneko-Zagier's conjecture. It implies the existence of a surjective map $\red_{\zeta}$ given by reduction modulo large primes from the $\mathbb{Q}$-algebra of $p$-adic multiple zeta values in $\prod_{p} \mathbb{Z}_{p}$ onto the $\mathbb{Q}$-algebra of finite multiple zeta values. Kaneko-Zagier's conjecture reduces to the usual period conjecture on complex and $p$-adic multiple zeta values combined to :

\begin{Conjecture} \label{sept point trente et un}The map $\red_{\zeta}$ is an isomorphism. 
\end{Conjecture}

\noindent In our language this can be interpreted by saying that this is a conjecture of periods for finite multiple zeta values from the $\DR$ setting.

\subsection{Conjectures of periods beyond the previous ones\label{last}}

In II-1, Appendix A, we have obtained standard algebraic relations for finite multiple zeta values (double shuffle relations) by reduction of the standard algebraic relations for prime weighted multiple harmonic sums. This worked both from the $\DR$ and the $\RT$ setting. It is thus tempting to try to formulate three period conjectures for $\zeta_{\mathcal{A}}$ imitating respectively the three periods conjectures for $\har_{\mathcal{P}^{\mathbb{N}}}$ from \S7.1, \S7.2, \S7.3 that distinguish the $\DR$, $\DR-\RT$, and $\RT$ settings, and that would be formulated through the periods conjectures for sequences of prime weighted multiple harmonic sums ; thus, the $\DR$ conjecture would be different from the Conjecture \ref{sept point trente et un}.
\newline Independently, in \S7.1 we have said nothing about the maps sending algebras of prime weighted multiple harmonic sums equipped with the weight filtration to the same algebras equipped with the filtration given by the uniform topology on $\prod_{p\in\mathcal{P}}\mathbb{Z}_{p}$ associated with the $p$-adic topologies.
\newline 
\newline Conjectures of periods for $\zeta_{\mathcal{A}}$ in the sense above would indeed be related to the ones of \S7.1, \S7.2, \S7.3 by reduction modulo large primes ; they would imply statements on lifts of congruences modulo large $p$ between prime multiple harmonic sums. More generally, any conjecture going further the previous ones would be connected to some subtle arithmetic questions such as the valuation of prime weighted multiple harmonic sums and $p$-adic multiple zeta values, and the height of the denominators of the rational coefficients of algebraic relations between multiple zeta values. 
\newline As a consequence, we delay these questions to part III. We should arrive at a conjecture relating filtered algebras of prime weighted multiple harmonic sums and their associated graded algebras related to finite multiple zeta values.

\appendix

\section{Framework for obtaining variants of the standard algebraic relations}

\subsection{Introduction}

There are in the literature a lot of examples of relations between multiple zeta values proved by elementary algebraic of analytic methods applying to the iterated series and the iterated integral representation of multiple zeta values, and which have been later retrieved as consequences of the standard alebraic relations, in particular of the double shuffle relations.
\newline A significative example of work retrieving algebraic relations from the standard ones can be found in the paper \cite{Ihara Kaneko Zagier} of Ihara-Kaneko-Zagier (where Ihara means K.Ihara and not Y.Ihara who defined the so-called Ihara action). They prove that the double shuffle relations imply what they call a "derivation relation" ; and as a consequence, they retrieve several previously known relations ; among them, the "cyclic sum formula" for multiple zeta values, proved earlier by Hoffman and Ohno using partial fraction decompositions of iterated sums. They also give a list of automorphisms of the quasi-shuffle algebra. Their results suggest a more general method.
\newline In II-1, we showed that the map $\Sigma_{\inv}^{\DR}$ induced a map $\DMR_{0} \rightarrow \DMR^{\DR}_{\text{har,prime}}$ where $\DMR^{\DR}_{\text{har,prime}}$ was a certain formal scheme, and we viewed the equations of $\DMR^{\DR}_{\text{har,prime}}$ as standard relations for adjoint multiple zeta values and prime weighted multiple harmonic sums ; they can also be thought of as variants, non-standard algebraic relations for multiple zeta values. Independently, in II-1, Appendix B, we proved an analogue of the cyclic sum formula for prime weighted multiple harmonic sums.
\newline Although the philosophies are different, there are common computational features between the methods of Ihara-Kaneko-Zagier and our method that we used in II-1. In this first appendix, we explain what these common features are, and we discuss an essentialization of both methods.

\subsection{Transfer of double shuffle relations}

Let $\mathcal{O}^{\ast,e_{Z}}$ be the vector subspace of the shuffle Hopf algebra $\mathcal{O}^{\sh,e_{Z}}$, generated by words whose first letter at the right is not $e_{0}$. It is also the vector space underlying the quasi-shuffle Hopf algebra. We have the natural inclusion $\mathcal{O}^{\ast,e_{Z}} \hookrightarrow \mathcal{O}^{\sh,e_{Z}}$. On the other hand, we have the surjective linear map

$$ \begin{array}{c} \mathcal{O}^{\sh,e_{Z}} \twoheadrightarrow \mathcal{O}^{\ast,e_{Z}} \\ e_{0}^{s_{d}-1}e_{z_{i_{d}}} \ldots e_{0}^{s_{1}-1}e_{z_{i_{1}}} e_{0}^{l} \mapsto \sum_{\substack{l_{1},\ldots,l_{d} \geq 0 \\ l_{1}+\ldots+l_{d} = l}} \prod_{i=1}^{d} {-s_{i} \choose l_{i}} e_{0}^{s_{d}+l_{d}-1}e_{z_{i_{d}}} \ldots e_{0}^{s_{1}+l_{1}-1}e_{z_{i_{1}}} \end{array} 
$$

\noindent Given a linear map $f :\mathcal{O}^{\sh,e_{Z}} \rightarrow \widehat{\mathcal{O}^{\sh,e_{Z}}}$, we will denote by $f_{\ast}$ the composition of $f$ by those two maps :
$$ f_{\ast} : \mathcal{O}_{\ast,e_{Z}} \hookrightarrow \mathcal{O}^{\sh,e_{Z}}  \overset{f}{\rightarrow} \widehat{\mathcal{O}^{\sh,e_{Z}}} \twoheadrightarrow \widehat{\mathcal{O}^{\ast,e_{Z}}} $$
\noindent Given $f$, a linear map 
$\mathcal{O}^{\sh,e_{Z}} \rightarrow \widehat{\mathcal{O}^{\sh,e_{Z}}}$, or a bilinear map defined on
$\mathcal{O}^{\sh,e_{Z}} \times \mathcal{O}^{\sh,e_{Z}}$, the corresponding weight-adic completion will be denoted by $\hat{f}$, and similarly $\hat{f}_{\ast}$ will be the completion of $\hat{f_{\ast}}$ (in accordance with \S3). This definition is motivated by II-1, \S4.

\begin{Definition} Let a map $\trf :\mathcal{O}^{\sh,e_{Z}}
	\rightarrow \widehat{\mathcal{O}^{\sh,e_{Z}}} $. We say that $\Sigma_{\omega}$ is a good transfer of the double shuffle equations if there exists a couple of bilinear maps 
	$$ \left\{ \begin{array}{l} B_{\sh} :\mathcal{O}^{\sh,e_{Z}} \times\mathcal{O}^{\sh,e_{Z}} \rightarrow\mathcal{O}^{\sh,e_{Z}}\times\mathcal{O}^{\sh,e_{Z}}
	\\B_{\ast} : \mathcal{O}^{\ast,e_{Z}} \times \mathcal{O}^{\ast,e_{Z}} \rightarrow \mathcal{O}^{\ast,e_{Z}} \times \mathcal{O}^{\ast,e_{Z}}
	\end{array} \right. $$
	\noindent such that :
	\begin{equation} \label{eq:r1} \hat{\sh} \circ (\trf \otimes \trf) = \trf \circ \sh \circ B_{\sh} \end{equation}
	\begin{equation} \label{eq:r2} \ast \circ (\trf \otimes \trf) = \trf \circ \ast \circ \hat{B}_{\ast} 
	\end{equation}
\end{Definition}

\subsection{Transfer of relations connected to associator relations and Kashiwara-Vergne relations}

The following definition is motivated by II-1, \S5.

\begin{Definition}
	We say that a map $\trf :\mathcal{O}^{\sh,e_{Z}}
	\rightarrow \widehat{\mathcal{O}^{\sh,e_{Z}}}$ is a good transfer of relations coming from automorphisms of the variety if, for each linear map $\tau :\mathcal{O}^{\sh,e_{Z}} \rightarrow\mathcal{O}^{\sh,e_{Z}}$ induced by an algebraic automorphism of $\mathbb{P}^{1} - \{0,\mu_{N},\infty\}$, there exists a linear map 
	$$ L_{\tau} :\mathcal{O}^{\sh,e_{Z}} \rightarrow\mathcal{O}^{\sh,e_{Z}}$$
	satisfying :
	\begin{equation} \label{eq:r3} \hat{\tau} \circ \trf = \trf \circ \tau \circ L_{\tau} 
	\end{equation}
	\noindent and a similar definition with $\mathcal{O}(\pi_{1}^{\un,\DR}(\mathcal{M}_{0,5}),\omega_{\DR})$ and the algebraic automorphisms of $\mathcal{M}_{0,5}$.
\end{Definition}

\subsection{Transfer of the motivic Galois action }

Goncharov's coaction $\Delta^{\mot}$ (\S\ref{Goncharov}) can be considered a map $\mathcal{O}^{\sh,e_{Z}} \rightarrow\mathcal{O}^{\sh,e_{Z}} \otimes\mathcal{O}^{\sh,e_{Z}}$, which can be factorized in a natural way as 
$$ \Delta = (\id \otimes \sh) \circ \Delta_{T} $$
\noindent with $\Delta_{T} :\mathcal{O}^{\sh,e_{Z}} \rightarrow\mathcal{O}^{\sh,e_{Z}} \otimes T(\mathcal{O}^{\sh,e_{Z}})$, where $T$ denotes the tensor algebra, and $\sh : T(\mathcal{O}^{\sh,e_{Z}}) \rightarrow \mathcal{O}^{\sh,e_{Z}}$ is the direct sum of the shuffle products of all tensor components.
The following definition is motivated by \S\ref{coaction appl}.

\begin{Definition} We say that a map $\trf :\mathcal{O}^{\sh,e_{Z}}
	\rightarrow \widehat{\mathcal{O}^{\sh,e_{Z}}}$ is a good transfer of the motivic Galois coaction if there exists a linear map
	$$ L_{\Delta} : T(\mathcal{O}^{\sh,e_{Z}})\rightarrow T(\mathcal{O}^{\sh,e_{Z}}) $$ 
	\noindent such that we have 
	$$ \Delta \circ \trf = \sh \circ \tilde{\trf} \circ L_{\Delta} \circ \Delta_{T} $$
	\noindent with $\tilde{\trf} : T(\mathcal{O}^{\sh,e_{Z}}) \rightarrow T(\mathcal{O}^{\sh,e_{Z}})$ is a sum of tensor products of $\trf$ and $\id_{\mathcal{O}^{\sh,e_{Z}}}$.
\end{Definition}

\subsection{Solution of a transfer equation related to II-1 and \cite{Ihara Kaneko Zagier}}

In this paragraph, we want both to explain the relation between our method in II-1 and \cite{Ihara Kaneko Zagier}, and also to essentialize our computations of II-1 : thus, we solve a particular case of the equations of preservations defined above. A part of the results of II-1 come from the following fact.

\begin{Fact} For all words $w\in \mathcal{O}^{\sh,e_{Z}}$, we have : 
$$ \frac{1}{1+e_{0}}\text{ } \sh \text{ } w(e_{z_{0}},e_{z_{1}},\ldots,e_{z_{r}}) = 
w\big(\frac{1}{1+e_{0}}e_{z_{0}},\frac{1}{1+e_{0}}e_{z_{1}},\ldots,\frac{1}{1+e_{0}}e_{z_{r}} \big) \frac{1}{1+e_{0}} = \shft_{\ast}(w) \frac{1}{1+e_{0}} $$
\end{Fact}

\noindent where the last equality can be taken as a definition of $\shft_{\ast} : \mathcal{O}^{\ast,e_{Z}} \rightarrow \widehat{\mathcal{O}^{\ast,e_{Z}}}$. In particular one has the following formula for $\shft_{\ast}$, which is also found in \cite{Ihara Kaneko Zagier} :
\begin{equation}
\shft_{\ast}: w\mapsto \big(\frac{1}{1+e_{0}} \text{ } \sh \text{ } w \big)(1+e_{0})
\end{equation}

\noindent We can ask ourselves what happens if we replace $\frac{1}{1+e_{0}}$ by a more general element of $\widehat{O^{\sh,e_{Z}}}$. Thus let us consider the following equation on $x \in \widehat{O^{\sh,e_{Z}}}$
\begin{equation} \label{eq:eq} 
\text{For all words w }\in \mathcal{O}^{\sh,e_{Z}}, \text{ }
x \text{ }\sh \text{ } w(e_{0},e_{z_{1}},\ldots,e_{1}) = w(xe_{0},xe_{z_{1}},\ldots,xe_{1})x
\end{equation}

\begin{Proposition} The solutions of equation (\ref{eq:eq}) are the elements of the form
	$$ x = (1-x_{1})^{-1} = \sum_{s \geq 0} x_{1}^{s}$$ 
\noindent where $x_{1}$ is of weight $1$.
\end{Proposition}

\begin{proof} We take $w = e_{z_{i}}$ in (\ref{eq:eq}). We obtain 
	\begin{equation} \label{eq:eq W1}
	x \text{ } \sh \text{ } e_{z_{1}} = x \text{ }e_{z_{i}} \text{ } x
	\end{equation}
	Let $x\in \widehat{O^{\sh,e_{Z}}}$ satisfying (\ref{eq:eq W1}). For $k\in \mathbb{N}$, let $x_{k}$ be the weight $k$ part of $x$ ; we have $x = \sum_{k=0}^{\infty} x_{k}$. The term $x_{0}$ is of the form $\lambda_{0}.\emptyset$ where $\lambda_{0} \in \mathbb{Q}$ and $\emptyset$ is the empty word. Considering the part of weight one of (\ref{eq:eq W1}) gives $x_{0}e_{z_{i}} = x_{0}^{2}e_{z_{i}}$, thus, $x_{0} = 0$ or $x_{0} = 1$. We distinguish those two cases.
	\newline i) Assume $x_{0} = 0$. Recall Radford's theorem that the shuffle algebra $(\mathcal{O}^{\sh,e_{Z}},+,\sh)$ is a free polynomial algebra over Lyndon words over the letters $e_{z_{i}}$. In particular we have the implication : for all $y \in \mathcal{O}^{\sh,e_{Z}}$,  $y \text{ }\sh \text{ } e_{z_{i}} = 0 \Rightarrow y = 0$. Using this implication, by induction on $k$, considering the weight $k+1$ part of (\ref{eq:eq W1}), we obtain that $x_{0} = \ldots = x_{k-1} = 0$.
	\newline ii) If $x_{0} = 1$, considering the weight $k+1$ part of  (\ref{eq:eq W1}), we obtain by induction on $k$, that $x_{k}$ is both of the form $x_{1}^{k-1}u$ and $u'x_{1}^{k-1}$, hence $x_{k} = x_{1}^{k}$. This implies $x = (1-x_{1})^{-1}$. 
	\newline Conversely, a word satisfying $x_{k} = x_{1}^{k}$ for all $k \in \mathbb{N}$ clearly satisfies the equation.
\end{proof}

\begin{Remark} The first interpretation of this result is that there is not a large generalization of the computations of II-1 under their precise form.
\newline At the same time, this result means that much of the combinatorics of part I remain true when replacing $(1-\Lambda e_{0})^{-1}$ by $(1-\sum_{z \in Z - \{\infty\}} \Lambda_{z_{i}} e_{z_{i}} )^{-1}$ where $\Lambda_{z_{i}}$ are formal variables.
\newline We have 
	$$ (1 - \Lambda_{0}e_{0})(1 - \sum_{z \in Z - \{\infty\}} \Lambda_{z} e_{z})^{-1}  =
	(1 - \sum_{z \in Z - \{\infty\}} \Lambda_{z} e_{z} + \sum_{z \in Z - \{0,\infty\}} \Lambda_{z} e_{z}) (1 - \sum_{z \in Z - \{\infty\}} \Lambda_{z} e_{z})^{-1} $$
	$$ = 1 + \sum_{z \in Z - \{0,\infty\}} \Lambda_{z} e_{z} (1 - \sum_{z \in Z - \{\infty\}} \Lambda_{z} e_{z})^{-1} $$
	\noindent whence :
	$$ (1 - \sum_{z \in Z - \{\infty\}} \Lambda_{z} e_{z})^{-1} = (1 - \Lambda_{0}e_{0})^{-1}  + (1 - \Lambda_{0}e_{0})^{-1}\sum_{z \in Z - \{0,\infty\}} \Lambda_{z} e_{z} (1 - \sum_{z \in Z - \{\infty\}} \Lambda_{z} e_{z})^{-1} $$
	\noindent In other terms, the multiplication by $(1-\sum_{z \in Z - \{\infty\}} \Lambda_{z_{i}} e_{z_{i}} )^{-1}$ factorizes in a natural way by the multiplication by $(1-\Lambda e_{0})^{-1}$. Actually, an equality of this type has appeared in II-1, \S5, and was used to write the counterpart for prime weighted multiple harmonic sums of the one dimensional part of Kashiwara-Vergne equations. 
	\newline This suggests that replacing $(1-\Lambda e_{0})^{-1}$ by $(1-\sum_{z \in Z - \{\infty\}} \Lambda_{z_{i}} e_{z_{i}} )^{-1}$ does not really provide new examples.
\end{Remark}

\subsection{Towards another example}

\noindent In another work we will interpret Kawashima's relation for multiple zeta values in a Galois-theoretic way as the complex analogue of certain facts of this theory. Kawashima's relation follows from algebraic relations between the complex Newton series of multiple harmonic sums combined to an expression of the coefficients of the series expansion at $0$ of these Newton series in terms of multiple zeta values. These imply algebraic relations between multiple zeta values, and Kawashima also conjectures that these relations generate all algebraic relations among multiple zeta values. Kawashima's relation looks like a variant of the double shuffle relation and has been related to the derivation relation of \cite{Ihara Kaneko Zagier}, and thus to the double shuffle relations, by Tanaka.

\section{Remarks on the transcendence of $\har_{\mathcal{P}^{\mathbb{N}}}$}

We want to explain two things : first, that one can sometimes solve easily the problem of the transcendence of the sequences of algebraic numbers $\har_{\mathcal{P}^{\mathbb{N}}}$ ; secondly, that in the "completed" setting defined in \S3 and \S7.1, the question of the transcendence should maybe be played by a slightly more general and perhaps more natural question involving infinite sums (Remark \ref{la remarque}).

\subsection{Introduction}

Whereas the question of the transcendence of multiple zeta values is unaccessible, we are going to see that the one of the transcendence of sequences $\har_{p^{\mathbb{N}}}(w)$, $\har_{\mathcal{P}^{\alpha}}(w)$, $\har_{\mathcal{P}^{\mathbb{N}}}(w)$ is sometimes accessible, and this will rely on the following fact :

\begin{Fact} \label{fact fact}
	Let an infinite countable product $\prod_{n \in \mathbb{N}} A_{n}$ of $\mathbb{Q}$-algebras, such that either for all $n\in \mathbb{N}$, $A_{n} \subset \mathbb{C}$, or for all $n \in \mathbb{N}$, $A_{n} \in \mathbb{C}_{p}$. An element $a \in \prod_{n \in \mathbb{N}} A_{n}$ is transcendental as soon as it has infinitely many components that are pairwise distinct.
	Indeed, a polynomial $P\in \mathbb{Q}[T]$ such that $P(a)=0$ has then infinitely many roots, and is thus equal to $0$.
\end{Fact}

\begin{Remark} \label{la remarque}Since finite multiple zeta values $\zeta_{\mathcal{A}}(w)$ are obtained by reduction modulo primes from $\har_{\mathcal{P}^{\alpha}}(w)$, the transcendence of finite multiple zeta values may one day use information on $\har_{\mathcal{P}^{\alpha}}(w)$.
	\end{Remark}
	
\begin{Remark} The fact that the transcendence of  $\har_{p^{\mathbb{N}}}(w)$, $\har_{\mathcal{P}^{\alpha}}(w)$, $\har_{\mathcal{P}^{\mathbb{N}}}(w)$ is in certain cases a simple remark means heuristically that the nature of these objects is closer to the one of hyperlogarithms, that are analytic functions on $\mathbb{P}^{1} - Z$ and for which the transcendence is an simple known fact, rather than to the one of multiple zeta values.
	\end{Remark}

\begin{Remark} It seems however that, because of the nature of $\har_{\mathcal{P}^{\mathbb{N}}}$ as periods of the \emph{completed} continuous $\pi_{1}^{\un,\DR}(X_{K})^{\widehat{\text{cont}}}$ that we found in \S7.1, a more natural question, concerning these numbers, than their transcendence is to know whether one can have 
$$ F(\har_{\mathcal{P}^{\mathbb{N}}}) = 0 $$
\noindent for $F \in \mathbb{Q}[[H]]]$ a non-trivial formal power series such that $F(\har_{p^{\alpha}})$ is an absolutely convergent series in $\mathbb{Q}_{p}$ for all primes $p$ and $\alpha \in \mathbb{N}^{\ast}$.
\end{Remark}

\subsection{The case of $\mathbb{P}^{1} - \{0,1,\infty\}$}

\begin{Proposition} \label{first proposition of B}Any infinite subsequence of a sequence $\har_{\mathcal{P}^{\mathbb{N}}}(w)$ of $\mathbb{P}^{1} - \{0,1,\infty\}$ is transcendental.
\end{Proposition}

\begin{proof} We will view such sequences as elements of $\prod_{I} \mathbb{Q} \subset \prod_{I} \mathbb{C}$ where $I$ is an infinite subset of $\mathcal{P}^{\mathbb{N}}$. Let us fix an index $(s_{d},\ldots,s_{1})$. The map $n \mapsto \har_{n}(s_{d},\ldots,s_{1})$ is a strictly increasing function $\mathbb{N}^{\ast} \rightarrow \mathbb{R}$ ; it is in particular injective. Whence the result by Fact \ref{fact fact}.
\end{proof}

\subsection{The (strictly positive) totally real case}

\noindent Let us fix an embedding $\overline{\mathbb{Q}} \hookrightarrow \mathbb{C}$.

\begin{Proposition} Let $z_{r},\ldots,z_{1}$ be non-zero algebraic numbers whose images by this embedding are in $\mathbb{R}^{+\ast}$. For any $w=\big( \begin{array}{c} z_{i_{d+1}},\ldots,z_{i_{1}} \\ s_{d},\ldots,s_{1} \end{array}
	\big)$, with $s_{d},\ldots,s_{1} \in \mathbb{N}^{\ast}$, $i_{1},\ldots,i_{d} \in \{1,\ldots,r\}$, any infinite subsequence of the sequence $\har_{\mathcal{P}^{\mathbb{N}}}(w)$, in $\prod_{p} \overline{\mathbb{Q}} \subset \prod_{p} \overline{\mathbb{Q}_{p}}$ is transcendental.
\end{Proposition}

\begin{proof} The map $n \in \mathbb{N}^{\ast} \mapsto \har_{n}[w]$ is again injective ; same argument with Proposition \ref{first proposition of B}.
\end{proof}

\noindent There should be variants of this arguments for other subcases of the totally real case (for example, the case where all images of $z_{1},\ldots,z_{r}$ are in $\mathbb{R}^{-\ast}$).

\subsection{The universal case}

In this paragraph, we consider not a single curve $\mathbb{P}^{1} - Z$, but the fibrations $\mathcal{M}_{0,n+1} \rightarrow \mathcal{M}_{0,n}$ which admit these curves as fibers, i.e. we replace the variables $z_{1},\ldots,z_{r}$ by formal variables. Let, for $d \in \mathbb{N}^{\ast}$, and $(s_{d},\ldots,s_{1}) \in (\mathbb{N}^{\ast})^{d}$, and $n \in \mathbb{N}^{\ast}$,
	$$ P_{s_{1},\ldots,s_{d},n}(T_{1},\ldots,T_{d},T_{d+1}) = \sum_{0<n_{1}<\ldots<n_{d}<n} \frac{T_{1}^{n_{1}} \ldots  T_{d}^{n_{d}}}{n_{1}^{s_{1}} \ldots n_{d}^{s_{d}}} T_{d+1}^{n} \in \mathbb{Q}[T_{1},\ldots,T_{d},T_{d+1}] $$

\begin{Remark} The prime harmonic double shuffle relations from Theorem II-1.a lift to equalities between these polynomials for $n$ equal to a power of a prime number.
\end{Remark}

\begin{Proposition} \label{previous previous}
	For each $s_{d},\ldots,s_{1} \in \mathbb{N}^{\ast}$, any infinite subsequence of the sequences 
	\newline $(P_{s_{1},\ldots,s_{d},n}(T_{1},\ldots,T_{d},T_{d+1}))_{n\in\mathbb{N}}$ is transcendental.
\end{Proposition}

\begin{proof} The polynomials
	$P_{s_{1},\ldots,s_{d},n}(T_{1},\ldots,T_{d},T_{d+1})$ for $n \in \mathbb{N}^{\ast}$ are pairwise distinct because, for each $n \in \mathbb{N}^{\ast}$, the degree of $P_{s_{1},\ldots,s_{d},n}(T_{1},\ldots,T_{d},T_{d+1})$ with respect to $T_{d+1}$ is equal to $n$. Whence the result by Fact \ref{fact fact}.
\end{proof}

\subsection{The general case $\mathbb{P}^{1} - \{z_{0},z_{1},\ldots,z_{r}\}$ : conjecture}

\noindent We denote by $\sigma$ the non-weighted multiple harmonic sums, i.e. we have $\har_{n} = n^{\weight(w)}\sigma_{n}(w)$). Let a curve $\mathbb{P}^{1} - \{z_{0},z_{1},\ldots,z_{r}\}$ (with $z_{0} = 0$) over a number field embedded in $\mathbb{C}$.

\begin{Remark} Let a word $w=\big( \begin{array}{c} z_{i_{d+1}},\ldots,z_{i_{1}} \\ s_{d},\ldots,s_{1} \end{array}
	\big)$, such that $s_{d} \geq 2$. If there are infinitely many $n$ for which $\sigma_{n}(w)$ takes the same value, then the limit $\lim_{n \rightarrow \infty}\sigma_{n}[w]$, which exists in $\mathbb{C}$ is a value of an hyperlogarithm at a tangential base-point, is an algebraic number. This contradicts the usual transcendence conjectures on values of hyperlogarithms at tangential base-points.
	\newline In the case where $s_{d} = 1$, we can make a similar reasoning using the asymptotic expansion of multiple harmonic sums when $n \rightarrow +\infty$.
\end{Remark}

\noindent We are thus led to think that for any word $w=\big( \begin{array}{c} z_{i_{d+1}},\ldots,z_{i_{1}} \\ s_{d},\ldots,s_{1} \end{array}
	\big)$, the preimages of a singleton in $\mathbb{C}$ by the maps $n \mapsto \sigma_{n}(\tilde{w})$ are finite ; this would imply in particular that there are infinitely many distinct values of $\sigma_{n}(\tilde{w})$, $n\in \mathbb{N}$, and thus, by fact \ref{fact fact}, that $\har_{\mathcal{P}^{\mathbb{N}}}(w)$ is transcendental.

\end{document}